\documentclass[10pt,leqno]{article}

\usepackage{amsmath,amssymb,amsthm,mathrsfs,dsfont}

\usepackage[margin=3cm]{geometry} 

\usepackage{titlesec,hyperref}

\usepackage{color}

\usepackage{fancyhdr}
\titleformat{\subsection}{\it}{\thesubsection.\enspace}{1pt}{}

\newtheorem{theo}{Theorem}[section]
\newtheorem{lemm}[theo]{Lemma}
\newtheorem{defi}[theo]{Definition}

\newtheorem{prop}[theo]{Proposition}

\numberwithin{equation}{section}

\allowdisplaybreaks 

\begin{document}
\title{The local well-posedness, global existence and ill-posedness for the fifth order Camassa-Holm model
	\hspace{-4mm}
}

\author{Xiaoxin $\mbox{Chen}^1$ \footnote{E-mail: chenxx233@mail2.sysu.edu.cn},\quad
	Zhaoyang $\mbox{Yin}^{1}$\footnote{E-mail: mcsyzy@mail.sysu.edu.cn}\\
$^1\mbox{School}$ of Science,\\ Shenzhen Campus of Sun Yat-sen University, Shenzhen 518107, China}

\date{}
\maketitle
\hrule

\begin{abstract}
In this paper, we consider the fifth order Camassa-Holm model. Firstly, we improve the local well-posedness results in \cite{TangLiu2015,FOCH2021}. Secondly, we give the blow up criteria and conditions for global existence. Finally, when $b=\frac 53$ in the model, we obtain the ill-posedness in $B^1_{\infty,1}$ and $B^{\frac 32}_{2,q}$ with $q\in(1,+\infty]$ in the sense of norm inflation.
	
	
\vspace*{5pt}
\end{abstract}

\vspace*{10pt}

\tableofcontents

\section{Introduction}
	
In this paper,  we focus on the Cauchy problem for the following fifth order Camassa-Holm (FOCH) equation \cite{LIUQIAO2018179}:
\begin{equation}\label{eq0}
\left\{\begin{array}{ll}
    m_t+u m_x+bu_x m=0,  \quad &t\in\mathbb{R}^{+},\ x\in\mathbb{R},  \\
    m=(1-\alpha^2\partial_{xx})(1-\beta^2\partial_{xx})u,  \quad &t\in\mathbb{R}^{+},\ x\in\mathbb{R},  \\
    m|_{t=0}=m_0,  \quad &x\in\mathbb{R},
\end{array}\right.
\end{equation}
where $\alpha,\beta,b\in\mathbb{R},\ \alpha\beta\neq0$. If we denote $P(D)=(1-\alpha^2\partial_{xx})^{-1}(1-\beta^2\partial_{xx})^{-1}$, then\ \eqref{eq0} can be rewritten as
\begin{equation}\label{eq1}
\left\{\begin{array}{ll}
    u_t+u u_{x}+F(u)=0,\quad &t\in\mathbb{R}^{+},\ x\in\mathbb{R},\\
    u|_{t=0}=u_{0},  \quad &x\in\mathbb{R},
\end{array}\right.
\end{equation}
where $F(u)=F_{1}(u)+F_{2}(u)+F_{3}(u)+F_{4}(u)$ with
\begin{equation}\label{eqF}
\left\{\begin{array}{ll}
    F_{1}(u)=\frac{b}{2}\partial_{x}P(D)\left(u^{2}\right), 
    & F_{2}(u)=\frac{3-b}{2}(\alpha^2+\beta^2)\partial_{x}P(D)\left(u_{x}^{2}\right),\\
    F_{3}(u)=\frac{5-3b}{2}\alpha^2\beta^2\partial_{x}P(D)\left(u_{xx}^{2}\right), 
    & F_{4}(u)=\frac{b-5}{2}\alpha^2\beta^2\partial_{x}^3P(D)\left(u_{x}^{2}\right).
\end{array}\right.
\end{equation}

If $\alpha=1,\ \beta=0$ (or $\alpha=0,\ \beta=1$) and $b=2$, \eqref{eq0} becomes the well-known Camassa-Holm (CH) equation, which is completely integrable and bi-Hamiltonian, and admits an infinite number of conserved quantities \cite{CH1993,Constantin1997,Constantin2001,Constantin2006,Fuchssteiner1981}. Applying the Littlewood–Paley theory and the transport equation theory, the local well-posedness of CH equation in Besov space $B^s_{p,r}$ with $s>\max\{1+\frac 1p, \frac 32\},\ r\in[1,+\infty)$ or $s=1+\frac 1p,\ p\in[1,2],\ r=1$ was investigated in \cite{Danchin2001,Danchin2003,LiYin2016}. Byers \cite{Byers2006ill} proved the ill-posedness in $H^s$ with $s<\frac 32$. There are many studies on the global existence and blow-up phenomena of CH equation, such as \cite{Constantin2000blow,Constantin1998globalblow,JiangNiZhou2012blow}. Recently, Guo et al. \cite{GuoLiuMolinetYin2019illCH} proved that the Camassa–Holm and related equations is ill-posed in $B^{1+\frac 1p}_{p,r}$ for $p\in[1,+\infty],\ r\in(1,+\infty]$. Li et.al \cite{Li2022CHill} established the ill-posedness in $B^s_{p,\infty}$ with $s>\max\{1+\frac 1p, \frac 32\}$. Ye et al. \cite{GUOLinJieShiDing} established the local well-posedness in $B^{1+\frac 1p}_{p,1}$ with $p\in[1,+\infty)$.  Guo et al. \cite{GUObushiding} gave the ill-posedness for the CH equation in $B^1_{\infty,1}$. In addition, when $\alpha=1,\ \beta=0$ (or $\alpha=0,\ \beta=1$), the Degasperis-Procesi equation and the b-equation are the special cases of \eqref{eq0} with $b=3$ and $b\in\mathbb{R}$, respectively. The local well-posedness, ill-posedness, global existence and blow-up phenomena were discussed in \cite{Coclite2006DP,EscherYin2008,GuiLiu2011DP,GuiLiu2008bfamily,GUOSharpill,HHG2014DPill,LiuYin2006DP,LiuYin2007DP}.

Recently, there have been explorations of the FOCH model \eqref{eq0}.
When $\alpha\neq\beta$, Liu and Qiao \cite{LIUQIAO2018179} studied peakon solutions including single pseudo-peakons, two-peakon and three-peakon interactional solutions, and N-peakon dynamical system. Zhu et al. \cite{FOCH2021} established the local well-posedness in $H^s$ with $s>\frac 72$ and global existence of the solution. The property of the infinite propagation speed and the long time behavior for the support of momentum density were also discussed in \cite{FOCH2021}. When $\alpha=\beta=1,\ b=2$, Tang and Liu \cite{TangLiu2015} established the local well-posedness in the critical space $B^{\frac 72}_{2,1}$ or in $B^s_{p,r}$ with $s>\max\{3+\frac 1p, \frac72\},\ r\in[1,+\infty)$, and obtained the peakon-like solution which guaranteed the ill-posedness in $B^{\frac 72}_{2,\infty}$. For the periodic case, McLachlan
and Zhang \cite{Robertfoch2202540circle} showed that the solution is locally well-posed in $H^s(\mathbb{T})$ with $s>\frac 72$ and established the global existence of the solution. Fu and Liu \cite{FuLiuFOCH2013pad} proved the non-uniform dependence on initial data in $H^s(\mathbb{T})$ with $s>\frac 72$. There are also studies on other higher-order equations of Camassa-Holm type, which have been presented in \cite{CocliteHoldenKarlsen2009,DingD2017,DingD2010,Robertfoch2202540circle,QiaoReyesfoch2,TianZhangXia2011,WangLiQiao2018}. 

Although there have been studies on the local well-posedness of the FOCH model \eqref{eq0} as described above, whether the model is well-posed or ill-posed in $B^s_{p,r}$ with $s<\max\{3+\frac 1p, \frac 72\}$ remains open. In this paper, taking advantage of the Moser-type inequality, we establish the local well-posedness in $B^s_{p,r}$ with $s>\max\{2+\frac 1p, \frac 52\}\ (resp.,\ s>\max\{1+\frac 1p, \frac 32\}),\ r\in[1,+\infty),\ p\in[1,+\infty]$ when $b\neq\frac 53\ (resp.,\ b=\frac 53)$, see Theorem \ref{the1}. Furthermore, using the compactness method and the Lagrange coordinate transformation, we obtain the local well-posedness in $B^{2+\frac 1p}_{p,1}\ (resp.,\ B^{1+\frac 1p}_{p,1})$ with $p\in[1,+\infty)$ when $b\neq\frac 53\ (resp.,\ b=\frac 53)$, which implies the index $\frac 52\ (resp.,\ \frac 32)$ is not necessary, see Theorem \ref{the2}. When $0\leq b\leq1$, we present a new global existence result by using a new conservation law $\|m\|_{L^{\frac 1b}}$, see Theorem \ref{global}. For the ill-posed problems, by defining the norm $\|\cdot\|_{B^0_{\infty,\infty,1}}$ and providing some relevant properties, we prove the ill-posedness in $B^1_{\infty,1}$ when $b=\frac 53$ in the sense of norm inflation, see Theorem \ref{ill1}. In addition, when $b=\frac 53$, we give the ill-posedness in $B^{\frac 32}_{2,q}$ with $q\in(1,+\infty]$, the proof of which uses the blow up criteria but does not rely on the blow-up results, see Theorem \ref{ill2}.

The paper is organized as follows. In Section \ref{Preliminaries}, we state some preliminaries which will be used in the sequel. In Section \ref{LOcal}, we improve the local well-posedness results in \cite{TangLiu2015,FOCH2021}. In Section \ref{Global existence}, we give the blow up criteria for \eqref{eq1} and conditions for global existence. In Section \ref{Ill-posedness}, we prove that when $b=\frac 53$, the Cauchy problem \eqref{eq1} is ill-posed in $B^1_{\infty,1}$ and $B^{\frac 32}_{2,q}$ with $q\in(1,+\infty]$ in the sense of norm inflation.
	
\section{\textbf{Preliminaries}}\label{Preliminaries}
	
In this section, we first present some facts on the Littlewood-Paley decomposition and nonhomogeneous Besov spaces.

\begin{prop}\cite{BCD}\label{Bernstein’s inequalities}
Let $\mathscr{C}$ be an annulus and $\mathscr{B}$ a ball. A constant C exists such that for any $k\in\mathbb{N}$, $1\leq p\leq q\leq\infty$, and any function $u$ of $L^p(\mathbb{R}^d)$, we have
$$
\mathrm{Supp~}\widehat{u}\subset\lambda\mathscr{B}\Longrightarrow\|D^ku\|_{L^q}=\sup_{|\alpha|=k}\|\partial^\alpha u\|_{L^q}\leq C^{k+1}\lambda^{k+d(\frac{1}{p}-\frac{1}{q})}\|u\|_{L^p},
$$
$$
\mathrm{Supp~}\widehat{u}\subset\lambda\mathscr{C}\Longrightarrow C^{-k-1}\lambda^k\|u\|_{L^p}\leq\|D^ku\|_{L^p}\leq C^{k+1}\lambda^k\|u\|_{L^p}.
$$
\end{prop}

\begin{prop}\cite{BCD}\label{prop0}
Let $\mathcal{B}=\{\xi\in\mathbb{R}^d: |\xi|\leq \frac 43\}$ and $\mathcal{C}=\{\xi\in\mathbb{R}^d: \frac 34 \leq|\xi|\leq \frac 83\}$. There exist two smooth, radial functions $\chi$ and $\varphi$, valued in the interval $[0,1]$, belonging respectively to $\mathcal{D(B)}$ and $\mathcal{D(C)}$, such that
$$
	\forall\ \xi\in\mathbb{R}^{d},\ \chi(\xi)+\sum_{j\geq0}\varphi(2^{-j}\xi)=1, 
$$
$$
	\forall\ \xi\in\mathbb{R}^{d}\setminus\{0\},\ \sum_{j\in\mathbb{Z}}\varphi(2^{-j}\xi)=1, 
$$
$$
	|j-j^{\prime}|\geq2\Rightarrow\mathrm{Supp~}\varphi(2^{-j}\cdot)\cap\mathrm{Supp~}\varphi(2^{-j^{\prime}}\cdot)=\emptyset, 
$$
$$
	j\geq1\Rightarrow\mathrm{Supp~}\chi\cap\mathrm{Supp~}\varphi(2^{-j}\cdot)=\emptyset,
$$
the set $\widetilde{\mathcal{C}}\triangleq B(0,\frac 23)+\mathcal{C}$ is an annulus, and we see
$$
|j-j^{\prime}|\geq5\Rightarrow2^{j^{\prime}}\widetilde{\mathcal{C}}\cap2^j\mathcal{C}=\emptyset.
$$
Let $u\in\mathcal{S}'$. Defining
$$
\Delta_{j}u\triangleq0\ \text{if}\ j\leq-2,\ \Delta_{-1}u\triangleq\chi(D)u=\mathcal{F}^{-1}(\chi\mathcal{F}u),
$$
$$
\Delta_{j}u\triangleq\varphi(2^{-j}D)u=\mathcal{F}^{-1}(\varphi(2^{-j}\cdot)\mathcal{F}u)\ \text{if}\ j\geq0,
$$
$$
S_j u\triangleq\sum_{j^{\prime}\leq j-1}\Delta_{j^{\prime}}u,
$$
we have the following Littlewood-Paley decomposition:
$$u=\sum_{j\in\mathbb{Z}}\Delta_{j}u\quad in\ \mathcal{S}'.$$
\end{prop}

\begin{defi}\cite{BCD}\label{Besov Space}
Let $s\in\mathbb{R}$ and $(p,r)\in[1,\infty]^2$. The nonhomogeneous Besov
space $B^s_{p,r}$ consists of all $u\in\mathcal{S}'$ such that
$$
\|u\|_{B_{p,r}^s}\triangleq\left\|(2^{js}\|\Delta_ju\|_{L^p})_{j\in\mathbb{Z}}\right\|_{\ell^r(\mathbb{Z})}<\infty.
$$
\end{defi}

\begin{defi}\cite{BCD}\label{Bony decomposition}
Considering $u,v\in\mathcal{S}'$, we have the following Bony decomposition:
$$uv=T_u v+T_v u+R(u,v),$$
where
$$
T_u v=\sum_j S_{j-1}u\Delta_j v,\ R(u,v)=\sum_{|k-j|\leq1}\Delta_ku\Delta_jv.
$$
\end{defi}

\begin{prop}\cite{BCD}\label{prop1}
Let $s,s'\in\mathbb{R},\ 1\leq p,p_1,p_2,r,r_1,r_2\leq\infty.$  \\
(1) $B^s_{p,r}$ is a Banach space, and is continuously embedded in $\mathcal{S}'$. \\
(2) If $r<\infty$, then $\lim\limits_{j\rightarrow\infty}\|S_j u-u\|_{B^s_{p,r}}=0$. If $p,r<\infty$, then $\mathcal{C}_0^{\infty}$ is dense in $B^s_{p,r}$. \\
(3) If $p_1\leq p_2$ and $r_1\leq r_2$, then $ B^s_{p_1,r_1}\hookrightarrow B^{s-d(\frac 1 {p_1}-\frac 1 {p_2})}_{p_2,r_2}$. If $s<s'$, then the embedding $B^{s'}_{p,r_2}\hookrightarrow B^{s}_{p,r_1}$ is locally compact. \\
(4) $B^s_{p,r}\hookrightarrow L^{\infty} \Leftrightarrow s>\frac d p\ \text{or}\ s=\frac d p,\ r=1$. \\
(5) If $s\in\mathbb{R}^+\backslash\mathbb{N}$, then $B^s_{\infty,\infty}$ coincides with the Hölder space $\mathcal{C}^{[s],s-[s]}$. If $s\in\mathbb{N}$, then $B^s_{\infty,\infty}$ is strictly larger than the space $\mathcal{C}^{s}\ ($and than $\mathcal{C}^{s-1,1}$, if $s\in\mathbb{N}^\ast)$.\\
(6) Fatou property: if $(u_n)_{n\in\mathbb{N}}$ is a bounded sequence in $B^s_{p,r}$, then an element $u\in B^s_{p,r}$ and a subsequence $(u_{n_k})_{k\in\mathbb{N}}$ exist such that
$$\lim_{k\rightarrow\infty}u_{n_k}=u\ \text{in}\ \mathcal{S}'\quad \text{and}\quad \|u\|_{B^s_{p,r}}\leq C\liminf_{k\rightarrow\infty}\|u_{n_k}\|_{B^s_{p,r}}.$$
(7) Let $m\in\mathbb{R}$ and $f$ be a $S^m$- multiplier, $(i.e.\ f$ is a smooth function and satisfies that $\forall\ \alpha\in\mathbb{N}^d$, $\exists\ C=C(\alpha)$, such that $|\partial^{\alpha}f(\xi)|\leq C(1+|\xi|)^{m-|\alpha|},\ \forall\ \xi\in\mathbb{R}^d)$. Then the operator $f(D)=\mathcal{F}^{-1}(f\mathcal{F}\cdot)$ is continuous from $B^s_{p,r}$ to $B^{s-m}_{p,r}$.\\
(8) If $s<s'$, then for all $\phi\in \mathcal{S}(\mathbb{R})$, multiplication by $\phi$ is a compact operator from $B^{s'}_{p,\infty}$ to $B^{s}_{p,1}$.
\end{prop}

\begin{lemm}\cite{BCD}\label{log}
There are two useful interpolation inequalities:\\
(1) If $u\in B^{s_1}_{p,r}\cap B^{s_2}_{p,r}$, then for all $\theta\in[0,1]$, we have
$$ \|u\|_{B^{\theta s_1+(1-\theta)s_2}_{p,r}}\leq \|u\|_{B^{s_1}_{p,r}}^{\theta}\|u\|_{B^{s_2}_{p,r}}^{1-\theta}. $$
(2) Let $\varepsilon\in(0,1)$. A constant $C$ exists such that for any $f\in B^\varepsilon_{\infty,\infty}$,
$$
\|f\|_{L^{\infty}}\leq\frac{C}{\varepsilon}\|f\|_{B_{\infty,\infty}^{0}}\left(1+\log\frac{\|f\|_{B_{\infty,\infty}^{\varepsilon}}}{\|f\|_{B_{\infty,\infty}^{0}}}\right).
$$ 
\end{lemm}

\begin{prop}\cite{BCD}
Let $s\in\mathbb{R},\ 1\leq p,r\leq\infty.$
\begin{equation*}\left\{
	\begin{array}{l}
		B^s_{p,r}\times B^{-s}_{p',r'}\longrightarrow\mathbb{R},  \\
		(u,\phi)\longmapsto \sum\limits_{|j-j'|\leq 1}\langle \Delta_j u,\Delta_{j'}\phi\rangle,
	\end{array}\right.
\end{equation*}
defines a continuous bilinear functional on $B^s_{p,r}\times B^{-s}_{p',r'}$. Denoted by $Q^{-s}_{p',r'}$ the set of functions $\phi$ in $\mathcal{S}'$ such that$\|\phi\|_{B^{-s}_{p',r'}}\leq 1$. If $u$ is in $\mathcal{S}'$, then we have
$$\|u\|_{B^s_{p,r}}\leq C\sup_{\phi\in Q^{-s}_{p',r'}}\langle u,\phi\rangle.$$
\end{prop}

We have the following product laws:
\begin{lemm}\label{product}\cite{BCD}
The following estimates hold:\\
(1) For any $s>0$ and any $(p,r)$ in $[1,\infty]^2$, the space $L^{\infty} \cap B^s_{p,r}$ is an algebra, and a constant $C=C(s,d)$ exists such that
$$ \|uv\|_{B^s_{p,r}}\leq C(\|u\|_{L^{\infty}}\|v\|_{B^s_{p,r}}+\|u\|_{B^s_{p,r}}\|v\|_{L^{\infty}}). $$
(2) If $1\leq p,r\leq \infty,\ s_1\leq s_2,\ s_2>\frac{d}{p}\ (s_2 \geq \frac{d}{p}\ \text{if}\ r=1)$ and $s_1+s_2>\max\{0, \frac{2d}{p}-d\}$, there exists $C=C(s_1,s_2,p,r,d)$ such that
$$ \|uv\|_{B^{s_1}_{p,r}}\leq C\|u\|_{B^{s_1}_{p,r}}\|v\|_{B^{s_2}_{p,r}}. $$
\end{lemm}

\begin{lemm}\cite{BCD}\label{commutator}
Let $\theta\in C^1(\mathbb{R})$ such that $(1+|\cdot|)\widehat{\theta}\in L^1$. There exists a constant $C>0$ such that for any Lipschitz function $f$ with gradient in $L^p$ and any $g\in L^q$, we have, for any $\lambda>0$,
$$
\|[\theta(\lambda^{-1}D),f]g\|_{L^r}\leq C\lambda^{-1}\|\partial_{x} f\|_{L^p}\|g\|_{L^q}\quad with\quad\frac{1}{p}+\frac{1}{q}=\frac{1}{r}.
$$
\end{lemm}

Here is the Gronwall lemma.
\begin{lemm}\label{osgood}\cite{BCD}
Let $f(t), g(t)\in C^{1}([0,T])$ and $f(t), g(t)>0.$ Let $h(t)$ be a continuous function on $[0,T].$ Assume that, for all $t\in [0,T]$,
$$\frac 1 2 \frac{d}{dt}f^{2}(t)\leq h(t)f^{2}(t)+g(t)f(t).$$
Then for any time $t\in [0,T],$ we have
$$f(t)\leq f(0)\exp\int_0^t h(\tau)d\tau+\int_0^t g(\tau)\exp\left(\int_\tau ^t h(t')dt'\right)d\tau.$$
\end{lemm}

In the paper, we also need some results about the following transport equation:
\begin{equation}\label{transport}
	\left\{\begin{array}{ll}
		\partial_{t}f+v\partial_x f=g,&\ x\in\mathbb{R},\ t>0, \\
		f|_{t=0}=f_0,&\ x\in\mathbb{R}.
	\end{array}\right.
\end{equation}

\begin{lemm}\label{priori estimate}\cite{BCD,GUObushiding}
Let $1\leq p,r\leq\infty,\ s>-\min\{\frac{1}{p},\frac{1}{p'}\}$.There exists a constant $C$ such that for all solutions $f\in L^{\infty}([0,T];B^s_{p,r})$ of \eqref{transport} with initial data $f_0\in B^s_{p,r}$, and $g\in L^1([0,T];B^s_{p,r})$, we have, for a.e. $t\in[0,T]$,
$$ \|f(t)\|_{B^s_{p,r}}\leq \|f_0\|_{B^s_{p,r}}+\int_0^t \|g(t')\|_{B^s_{p,r}}dt'+C\int_0^t V^{\prime} (t')\|f(t')\|_{B^s_{p,r}}dt{'} $$
or
$$ \|f(t)\|_{B^s_{p,r}}\leq e^{CV(t)}\Big(\|f_0\|_{B^s_{p,r}}+\int_0^t e^{-CV(t')}\|g(t')\|_{B^s_{p,r}}dt'\Big) $$
with
\begin{equation*}
V^{\prime}(t)=\left\{
\begin{array}{ll}
     \|\partial_xv(t)\|_{B_{p,\infty}^{\frac{1}{p}}\cap L^\infty}, & \text{if} \ s<1+\frac{1}{p}, \\
     \|\partial_xv(t)\|_{B_{p,r}^s}, & \text{if}\ s=1+\frac{1}{p}\ and\ r>1, \\
     \|\partial_xv(t)\|_{B_{p,r}^{s-1}}, & \text{if}\ s>1+\frac{1}{p}\ or\ \left\{s=1+\frac{1}{p}\ and\ r=1\right\}.
\end{array}\right.
\end{equation*}
In particular, if $f=v$, then for all $s>0$, $V^{\prime}(t)=\|\partial_x  v(t)\|_{L^{\infty}}.$
\end{lemm}
	
\begin{lemm}\label{existence}\cite{BCD,GUObushiding}
Let $1\leq p, r\leq\infty,\ s> -\min\{\frac 1 p, \frac 1 {p'}\}$. Suppose $f_0\in B^s_{p,r}$, $g\in L^1([0,T];B^s_{p,r})$. Let $v$ be a time-dependent vector field such that $v\in L^\rho([0,T];B^{-M}_{\infty,\infty})$ for some $\rho>1$ and $M>0$, and
	$$
	\begin{array}{ll}
		\partial_x v\in L^1([0,T];B^{\frac 1 p}_{p,\infty}\cap L^\infty), &\ \text{if}\ s<1+\frac 1 p, \\
		\partial_x v\in L^1([0,T];B^{s}_{p,r}), &\ \text{if}\ s=1+\frac 1 p,\ r>1, \\
		\partial_x v\in L^1([0,T];B^{s-1}_{p,r}), &\ \text{if}\ s>1+\frac 1 {p}\ or\ \left\{s=1+\frac{1}{p}\ and\ r=1\right\}.
	\end{array}
	$$
	Then the equation \eqref{transport} has a unique solution $f$ in \\
	-the space $\mathcal{C}([0,T];B^s_{p,r})$, if $r<\infty$; \\
	-the space $\Big(\underset{s'<s}\bigcap \mathcal{C}([0,T];B^{s'}_{p,\infty})\Big)\bigcap \mathcal{C}_w([0,T];B^s_{p,\infty})$, if $r=\infty$.
\end{lemm}

\begin{lemm}\label{zp19}\cite{LiJinLuzp19}
Let $1\leq p\leq\infty,\ 1\leq r<\infty,\ s>1+\frac 1 p\ (or\ s=1+\frac 1 p,\ 1\leq p<\infty,\ r=1)$. Let $\overline{\mathbb{N}}=\mathbb{N}\cup\{\infty\}$. For $n\in\overline{\mathbb{N}}$, assume that $f^n\in \mathcal{C}([0,T];B^{s-1}_{p,r})$ is the solution to
\begin{equation*}
   	\left\{\begin{array}{l}
   		\partial_t f^n+a^n\partial_x f^n=g, \\
   		f^n|_{t=0}=f_0
 	\end{array}\right.
\end{equation*}
with $f_0\in B^{s-1}_{p,r},\ g\in L^1([0,T];B^{s-1}_{p,r})$. In addition, suppose that for some $\alpha\in L^1([0,T])$, 
$$\sup\limits_{n\in\overline{\mathbb{N}}}\|a^n(t)\|_{B^{s}_{p,r}}\leq \alpha(t).$$
If $a^n$ converges to $a^{\infty}$ in $L^1([0,T];B^{s-1}_{p,r})$, then $f^n$ converges to $f^{\infty}$ in $\mathcal{C}([0,T];B^{s-1}_{p,r})$.
\end{lemm}

We have the following estimate for $F(u)$ in \eqref{eq1}:
\begin{lemm}\label{Fu}
Suppose that $1\leq p,r\leq\infty,$
$$
\left\{\begin{array}{l}
s>\max\{\frac{5}{2},2+\frac{1}{p}\}\ or\ \{ s=2+\frac{1}{p},\ 1\leq p<\infty\ and\ r=1\}\quad \text{if} \quad b \neq \frac{5}{3}, \\
s>\max\{\frac{3}{2},1+\frac{1}{p}\}\ or\ \{ s=1+\frac{1}{p},\ 1\leq p<\infty\ and\ r=1\}\quad \text{if} \quad b = \frac{5}{3}.
\end{array}\right.
$$
Then $F$, given in \eqref{eq1}\eqref{eqF}, satisfies $F:B^s_{p,r}\to B^s_{p,r}$ and
$$\|F(u)\|_{B^{s}_{p,r}}\leq C\|u\|^2_{B^{s}_{p,r}}. $$
\end{lemm}
\begin{proof}
The Lemma can be easily proved by using Proposition \ref{prop1} and Lemma \ref{product}, so we omit the details here.
\end{proof}

\section{Local well-posedness}\label{LOcal}
In this section, we will establish the local well-posedness for the Cauchy problem \eqref{eq1} in certain Besov spaces. Before stating our results, we give the following definition:
\begin{defi}\label{espr}
For $s\in\mathbb{R},\ T>0,\ 1\leq p,r\leq\infty$,
$$E_{p,r}^{s}(T)\triangleq 
\left\{\begin{array}{l}
\mathcal{C}([0,T];B_{p,r}^{s}) \cap \mathcal{C}^{1}([0,T];B_{p,r}^{s-1}) \quad \text{if} \quad r < \infty, \\
\mathcal{C}_{w}(0,T;B_{p,\infty}^{s}) \cap \mathcal{C}^{0,1}([0,T];B_{p,\infty}^{s-1})\quad \text{if} \quad r =\infty.
\end{array}\right.
$$
\end{defi}

Firstly, we show the following theorem.
\begin{theo}\label{the1}
Let $u_{0}\in B_{p,r}^{s}$ with $1\leq p,r\leq\infty,$
\begin{align}\label{s}
\left\{\begin{array}{l}
	s>\max\{\frac{5}{2},2+\frac{1}{p}\} \quad \text{if} \quad b \neq \frac{5}{3}, \\
	s>\max\{\frac{3}{2},1+\frac{1}{p}\} \quad \text{if} \quad b = \frac{5}{3}.
\end{array}\right.
\end{align}
There exists a time $T>0$ such that\ \eqref{eq1} has a unique solution $u \in E_{p,r}^{s}(T)$. If $r<\infty\ (resp.,\ r=\infty)$, then the solution map is continuous from $B^{s}_{p,r}$ to $E^{s}_{p,r}(T)\ (resp.,\ \mathcal{C}_w([0,T];B^s_{p,\infty}))$.
\end{theo}
\begin{proof}
Now we prove Theorem \ref{the1} in three steps.
	
\textbf{Step 1: Existence}
	
Set $u^0\triangleq0$ and define a sequence $(u^n)_{n\in\mathbb{N}}$ of smooth functions by solving the following linear transport equation
\begin{equation}\label{un}
\left\{\begin{array}{l}
    \partial_t u^{n+1}+u^n\partial_xu^{n+1}+F(u^n)=0,\\
    u^{n+1}|_{t=0}=S_{n+1}u_{0}.
\end{array}\right.
\end{equation}
Since $u^0=0$, we have $u^1\equiv S_1 u_0\in E_{p,r}^{s}(T),\ \forall\ T>0.$ By induction, we assume that $u^n\in E_{p,r}^{s}(T)$ for all $T>0$. Then by lemma\ \ref{Fu}, $F(u^n)\in L^1([0,T];B^s_{p,r}),\ \forall\ T>0$. Combining Lemma\ \ref{existence} and \eqref{un}, we easily deduce that \eqref{un} has a solution $u^{n+1}\in E_{p,r}^{s}(T),\ \forall\ T>0$.

Define $U^n(t)=\int_0^t\|u^{n}(t')\|_{B^{s}_{p,r}}dt'$. Since $\|u^{n+1}(0)\|_{B^{s}_{p,r}} =\|S_{n+1}u_{0}\|_{B^{s}_{p,r}} \leq C\|u_0\|_{B^{s}_{p,r}} $, by Lemma \ref{priori estimate} and Lemma \ref{Fu}, we have 
    \begin{align}
	\|u^{n+1}(t)\|_{B^{s}_{p,r}} 
        &\leq C\|u_0\|_{B^{s}_{p,r}}e^{CU^n(t)}+\int_0^t e^{C\int_{t'}^t \|u^{n}(t'')\|_{B^{s}_{p,r}}dt''} \|F(u^{n})(t')\|_{B^{s}_{p,r}}dt'\notag\\
        &\leq C\Big(\|u_0\|_{B^{s}_{p,r}}e^{CU^n(t)}+\int_0^t e^{C\int_{t'}^t \|u^{n}(t'')\|_{B^{s}_{p,r}}dt''} \|u^{n}(t')\|_{B^{s}_{p,r}}^2dt'\Big).
        \label{18}
    \end{align}
For fixed  $T>0$ such that $ 2C^2 T\|u_0\|_{B^{s}_{p,r}} <1$, we have
    $$    \|u^1(t)\|_{B^{s}_{p,r}}\equiv\|S_1u_0\|_{B^{s}_{p,r}}\leq
	\frac{C\|u_0\|_{B^{s}_{p,r}}}{1-2C^2 t\|u_0\|_{B^{s}_{p,r}}},\ \forall\ t\in [0,T].
    $$
By induction, we assume that
    \begin{equation}\label{19}
	\|u^n(t)\|_{B^{s}_{p,r}}\leq
	\frac{C\|u_0\|_{B^{s}_{p,r}}}{1-2C^2 t\|u_0\|_{B^{s}_{p,r}}},\ \forall\ t\in [0,T].
    \end{equation}
Plugging \eqref{19} into \eqref{18}, we have
    \begin{align*}
	&\|u^{n+1}(t)\|_{B^{s}_{p,r}}\\
    \leq&
	C\|u_0\|_{B^{s}_{p,r}}(1-2C^2 t\|u_0\|_{B^{s}_{p,r}})^{-\frac 12}
	+C\int_0^t\Big(\frac{ 1-2C^2 t\|u_0\|_{B^{s}_{p,r}}}{1-2C^2 t'\|u_0\|_{B^{s}_{p,r}}}\Big)^{-\frac 12} \Big(\frac{C\|u_0\|_{B^{s}_{p,r}}}{1-2C^2 t'\|u_0\|_{B^{s}_{p,r}}}\Big)^2dt'  \\
    =&\frac{C\|u_0\|_{B^{s}_{p,r}}}{1-2C^2 t\|u_0\|_{B^{s}_{p,r}}},\ \forall\ t\in [0,T].
    \end{align*}
Therefore, $(u^n)_{n\in \mathbb{N}}$ is uniformly bounded in $L^{\infty}([0,T];B^{s}_{p,r})$, which implies that $(u^n\partial_x u^{n+1})_{n\in \mathbb{N}}$ is uniformly  bounded in $L^{\infty}([0,T];B^{s-1}_{p,r})$ and $(F(u^n))_{n\in \mathbb{N}}$ is uniformly  bounded in $L^{\infty}([0,T];B^{s}_{p,r})$. We can conclude that the sequence $(u^n)_{n\in \mathbb{N}}$ is uniformly bounded in $E^{s}_{p,r}(T)$.

We claim that $(u^n)_{n\in \mathbb{N}}$ is a Cauchy sequence in $\mathcal{C}([0,T];B^{s-1}_{p,r})$. Indeed, for all $(m,n)\in\mathbb{N}^2$,
\begin{equation}\label{cauchy}
    \left\{
    \begin{array}{l}
    (\partial_t+u^{n+m}\partial_x)(u^{n+m+1}-u^{n+1})=(u^n-u^{n+m})\partial_xu^{n+1}+F(u^n)-F(u^{n+m})\triangleq g^{n,m},\\
    (u^{n+m+1}-u^{n+1})|_{t=0}=S_{n+m+1}u_{0}-S_{n+1}u_{0}.
    \end{array}\right.
\end{equation}
Define 
$$
\widetilde{U}^{n,m}(t)=\left\{
\begin{array}{ll}
\|\partial_x u^{n+m}(t)\|_{B_{p,\infty}^{\frac{1}{p}}\cap L^\infty}, & \text{if} \ b=\frac{5}{3}\ and\ s<2+\frac{1}{p}, \\
\|\partial_x u^{n+m}(t)\|_{B_{p,r}^{s-1}}, & \text{if}\ b=\frac{5}{3}\ and\ s=2+\frac{1}{p},\ r>1, \\
\|\partial_x u^{n+m}(t)\|_{B_{p,r}^{s-2}}, & \text{if}\ b\in\mathbb{R}\ and\ s>2+\frac{1}{p}\ (or\ s=2+\frac{1}{p},\ r=1). 
\end{array}\right.
$$
Since $(u^n)_{n\in \mathbb{N}}$ is uniformly bounded in $L^{\infty}([0,T];B^{s}_{p,r})$, we have $\widetilde{U}^{n,m}(t)\leq C,\ \forall\ (m,n)\in\mathbb{N}^2,\ \forall\ t\in [0,T]$. Applying Lemma \ref{priori estimate} to \eqref{cauchy}, we have
\begin{align*}
&\|(u^{n+m+1}-u^{n+1})(t)\|_{B^{s-1}_{p,r}}\\
\leq\ &e^{C\int_0^t \widetilde{U}^{n,m}(t')dt'}\Big(\|S_{n+m+1}u_{0}-S_{n+1}u_{0}\|_{B^{s-1}_{p,r}}+\int_0^t  e^{-C\int_0^{t'}\widetilde{U}^{n,m}(t'')dt''}\|g^{n,m}(t')\|_{B^{s-1}_{p,r}}dt'\Big)      \\
\leq\ &C_T\Big(\|S_{n+m+1}u_{0}-S_{n+1}u_{0}\|_{B^{s-1}_{p,r}}+\int_0^t \|g^{n,m}(t')\|_{B^{s-1}_{p,r}}dt'\Big).
\end{align*}
Taking advantage of Proposition \ref{prop1} and Lemma \ref{product}, we can obtain
\begin{align*}
\|(u^n-u^{n+m})\partial_xu^{n+1}\|_{B^{s-1}_{p,r}}
&\leq C\|u^n-u^{n+m}\|_{B^{s-1}_{p,r}}\|\partial_xu^{n+1}\|_{B^{s-1}_{p,r}}\\
&\leq C_T\|u^n-u^{n+m}\|_{B^{s-1}_{p,r}},\\
\|F_1(u^n)-F_1(u^{n+m})\|_{B^{s-1}_{p,r}}
&\leq C\|(u^n-u^{n+m})(u^n+u^{n+m})\|_{B^{s-1}_{p,r}}\\
&\leq C_T\|u^n-u^{n+m}\|_{B^{s-1}_{p,r}},\\
\|F_2(u^n)-F_2(u^{n+m})\|_{B^{s-1}_{p,r}}&\leq C\|(u_x^n-u_x^{n+m})(u_x^n+u_x^{n+m})\|_{B^{s-2}_{p,r}}\\
&\leq C\|u_x^n-u_x^{n+m}\|_{B^{s-2}_{p,r}}\|u_x^n+u_x^{n+m}\|_{B^{s-1}_{p,r}}\\
&\leq C_T\|u^n-u^{n+m}\|_{B^{s-1}_{p,r}},\\
\|F_4(u^n)-F_4(u^{n+m})\|_{B^{s-1}_{p,r}}&\leq C\|(u_x^n-u_x^{n+m})(u_x^n+u_x^{n+m})\|_{B^{s-2}_{p,r}}\\
&\leq C_T\|u^n-u^{n+m}\|_{B^{s-1}_{p,r}}.
\end{align*}
If $b=\frac{5}{3}$, then $F_3=0$. If $b\neq\frac{5}{3},\ s>\max\{\frac{5}{2},2+\frac{1}{p}\}$, then by Lemma \ref{product}, 
\begin{align*}
\|F_3(u^n)-F_3(u^{n+m})\|_{B^{s-1}_{p,r}}&\leq C\|(u_{xx}^n-u_{xx}^{n+m})(u_{xx}^n+u_{xx}^{n+m})\|_{B^{s-3}_{p,r}}\\
&\leq C\|u_{xx}^n-u_{xx}^{n+m}\|_{B^{s-3}_{p,r}}\|u_{xx}^n+u_{xx}^{n+m}\|_{B^{s-2}_{p,r}}\\
&\leq C_T\|u^n-u^{n+m}\|_{B^{s-1}_{p,r}}.
\end{align*}
Since $\|S_{n+m+1}u_{0}-S_{n+1}u_{0}\|_{B^{s-1}_{p,r}}\leq C2^{-n}$, we have
\begin{align*}
&\|(u^{n+m+1}-u^{n+1})(t)\|_{B^{s-1}_{p,r}}\\
\leq\ &C_T\Big(2^{-n}+\int_0^t\|u^{n+m}(t_1)-u^n(t_1)\|_{B^{s-1}_{p,r}}dt_1\Big)\\
\leq\ &C_T\Big(2^{-n}+\int_0^tC_T\Big(2^{-n+1}+\int_0^{t_1}\|u^{n+m-1}(t_2)-u^{n-1}(t_2)\|_{B^{s-1}_{p,r}}dt_2\Big)dt_1\Big)\\
\leq\ &\dots\\
\leq\ &\sum_{l=1}^{n+1}C_T^l 2^{-n+l-1}\frac{T^{l-1}}{(l-1)!}+\int_0^t\int_0^{t_1}\dots\int_0^{t_n}C_T^{n+1}\|u^m(t_{n+1})\|_{B^{s-1}_{p,r}}dt_{n+1}\dots dt_2 dt_1\\
\leq\ &C_T 2^{-n}e^{2TC_T}+\frac{C_T^{n+2}T^{n+1}}{(n+1)!}\rightarrow 0\ (n\rightarrow\infty).
\end{align*}
Therefore, $(u^n)_{n\in \mathbb{N}}$ is a Cauchy sequence in $\mathcal{C}([0,T];B^{s-1}_{p,r})$ and converges to some limit function $u\in \mathcal{C}([0,T];B^{s-1}_{p,r}).$

Finally, we have to check that $u\in E_{p,r}^{s}(T)$ and satisfies \eqref{eq1}. Combining the uniform boundedness of $(u^n)_{n\in \mathbb{N}}$ in $L^{\infty}([0,T];B^{s}_{p,r})$ and the Fatou property, we can obtain that $u\in L^{\infty}([0,T];B^{s}_{p,r})$. Since $(u^n)_{n\in \mathbb{N}}$ converges to $u$ in $\mathcal{C}([0,T];B^{s-1}_{p,r})$, an interpolation argument ensures that convergence holds true in $\mathcal{C}([0,T];B^{s'}_{p,r}),\ \forall\ s'<s$. Thus, we can easily take the limit in \eqref{un} and conclude that $u$ is a solution of \eqref{eq1}. Since $u\in L^{\infty}([0,T];B^{s}_{p,r})$, Lemma \ref{Fu} implies that $F(u)\in L^{\infty}([0,T];B^{s}_{p,r})$. Hence, by Lemma \ref{existence}, we see that $u$ is in $ \mathcal{C}([0,T];B^s_{p,r})$ (resp., $\big(\underset{s'<s}\bigcap \mathcal{C}([0,T];B^{s'}_{p,\infty})\big)\bigcap \mathcal{C}_w([0,T];B^s_{p,\infty})$) if $r<\infty$ (resp., $r=\infty$). Thus, we deduce from \eqref{eq1} that $u_t$ belongs to $\mathcal{C}([0,T];B^{s-1}_{p,r})$ (resp., $L^{\infty}([0,T];B^{s-1}_{p,\infty})$) if $r<\infty$ (resp., $r=\infty$) and conclude that $u\in E_{p,r}^{s}(T)$.

\textbf{Step 2: Uniqueness}
	
Suppose $u,v$ are two solutions of \eqref{eq1} with initial data $u_0$ and $v_0$, respectively. Then
\begin{equation}\label{unique}
\left\{
\begin{array}{l}
(\partial_t+u \partial_x)(u-v)=(v-u)v_x+F(v)-F(u),\\
(u-v)|_{t=0}=u_0-v_0.
\end{array}
\right.    
\end{equation}
Define 
$$
\widetilde{U}(t)=\left\{
\begin{array}{ll}
\|\partial_x u(t)\|_{B_{p,\infty}^{\frac{1}{p}}\cap L^\infty}, & \text{if} \ b=\frac{5}{3}\ and\ s<2+\frac{1}{p}, \\
\|\partial_x u(t)\|_{B_{p,r}^{s-1}}, & \text{if}\ b=\frac{5}{3}\ and\ s=2+\frac{1}{p},\ r>1, \\
\|\partial_x u(t)\|_{B_{p,r}^{s-2}}, & \text{if}\ b\in\mathbb{R}\ and\ s>2+\frac{1}{p}\ (or\ s=2+\frac{1}{p},\ r=1). 
\end{array}\right.
$$
Similar to Step 1, applying Lemma \ref{priori estimate} to \eqref{unique} gives
\begin{align*}
&e^{-C\int_0^t \widetilde{U}(t')dt'}\|(u-v)(t)\|_{B^{s-1}_{p,r}}\\
\leq\ &\|u_{0}-v_{0}\|_{B^{s-1}_{p,r}}+\int_0^t  e^{-C\int_0^{t'}\widetilde{U}(t'')dt''}\big(\|(v-u)v_x\|_{B^{s-1}_{p,r}}+\|F(v)-F(u)\|_{B^{s-1}_{p,r}}\big)dt'      \\
\leq\ &\|u_{0}-v_{0}\|_{B^{s-1}_{p,r}}+C\int_0^t  e^{-C\int_0^{t'}\widetilde{U}(t'')dt''}\|u-v\|_{B^{s-1}_{p,r}}\big(\|u\|_{B^{s}_{p,r}}+\|v\|_{B^{s}_{p,r}}\big)dt' .
\end{align*}
Taking advantage of the Gronwall lemma, we have
$$
e^{-C\int_0^t \widetilde{U}(t')dt'}\|(u-v)(t)\|_{B^{s-1}_{p,r}}\leq\|u_{0}-v_{0}\|_{B^{s-1}_{p,r}}e^{C\int_0^t \|u\|_{B^{s}_{p,r}}+\|v\|_{B^{s}_{p,r}}dt'}.
$$
Since $\widetilde{U}(t)\leq C\|u(t)\|_{B^{s}_{p,r}}$, we obtain
$$
\|(u-v)(t)\|_{B^{s-1}_{p,r}}\leq\|u_{0}-v_{0}\|_{B^{s-1}_{p,r}}e^{C\int_0^t \|u\|_{B^{s}_{p,r}}+\|v\|_{B^{s}_{p,r}}dt'}.
$$
Hence, if $u_0=v_0$, then $u\equiv v$ and we complete the proof of uniqueness.

\textbf{Step 3: The continuous dependence}

Denote $\overline{\mathbb{N}}=\mathbb{N}\cup\{\infty\}$. Suppose that $u^n$ is the solution
of \eqref{eq1} with initial data $u_0^n\in B^{s}_{p,r}$ and that $u_0^n\to u_0^\infty$ in $B^{s}_{p,r}$ when $n\to\infty$.

\textbf{Case 1: $r<\infty$}

Fix $T>0$ such that $ 2C^2 T\underset{n\in\overline{\mathbb{N}}}{\sup}\|u_0^n\|_{B^{s}_{p,r}} <1$. Then by Step 1, we have
    $$    \underset{n\in\overline{\mathbb{N}}}{\sup}\|u^n\|_{L^\infty_T(B^{s}_{p,r})}\leq
	\frac{C^2\underset{n\in\overline{\mathbb{N}}}{\sup}\|u_0^n\|_{B^{s}_{p,r}}}{1-2C^2 T\underset{n\in\overline{\mathbb{N}}}{\sup}\|u_0\|_{B^{s}_{p,r}}}\triangleq M.
    $$
Similar to Step 2, we see that $u^n\to u^\infty$ in $\mathcal{C}([0,T];B^{s-1}_{p,r})$ when $n\to\infty$. Hence, we only need to prove $\partial_x u^n\to\partial_x u^\infty$ in $\mathcal{C}([0,T];B^{s-1}_{p,r})$. For simplicity, let $v^n\triangleq\partial_x u^n,\ f^n\triangleq-(\partial_x u^n)^2-\partial_x F(u^n)$ and decompose $v^n$ into $v^n=z^n+w^n$ with $(z^n,w^n)$ satisfying
\begin{equation}\label{zn}
    \left\{
    \begin{aligned}
	&\partial_t z^n+u^n\partial_xz^n=f^n-f^{\infty}, \\
	&z^n|_{t=0}=\partial_xu_0^n-\partial_xu_0^{\infty}\\
    \end{aligned}\right.
\end{equation}
and
\begin{equation}\label{wn}
    \left\{
    \begin{aligned}
	&\partial_t w^n+u^n\partial_x w^n=f^{\infty}, \\
	&w^n|_{t=0}=\partial_xu_0^{\infty}. \\
    \end{aligned}\right.
\end{equation}

Similar to Step 1 and Step 2, applying Lemma \ref{priori estimate} to \eqref{zn} gives
$$
\|z^n(t)\|_{B^{s-1}_{p,r}}
\leq e^{CMT}\Big(\|\partial_xu_{0}^n-\partial_xu_{0}^\infty\|_{B^{s-1}_{p,r}}+\int_0^t \|f^n(t')-f^{\infty}(t')\|_{B^{s-1}_{p,r}}dt'\Big).
$$
In addition, Proposition \ref{prop1} implies that
\begin{align*}
\|(u^n_x)^2-(u^{\infty}_x)^2\|_{B^{s-1}_{p,r}}
&\leq C\|u^n_x-u^{\infty}_x\|_{B^{s-1}_{p,r}}\|u^n_x+u^{\infty}_x\|_{B^{s-1}_{p,r}}\\
&\leq C\|u^n_x-u^{\infty}_x\|_{B^{s-1}_{p,r}},\\
\|\partial_x F_1(u^n)-\partial_x F_1(u^{\infty})\|_{B^{s-1}_{p,r}}&\leq C\|(u^n-u^{\infty})(u^n+u^{\infty})\|_{B^{s-1}_{p,r}}\\
&\leq C\|u^n-u^{\infty}\|_{B^{s-1}_{p,r}},\\
\|\partial_x F_2(u^n)-\partial_x F_2(u^{\infty})\|_{B^{s-1}_{p,r}}&\leq C\|(u_x^n-u_x^{\infty})(u_x^n+u_x^{\infty})\|_{B^{s-1}_{p,r}}\\
&\leq C\|u_x^n-u_x^{\infty}\|_{B^{s-1}_{p,r}},\\
\|\partial_x F_4(u^n)-\partial_x F_4(u^{\infty})\|_{B^{s-1}_{p,r}}&\leq C\|(u_x^n-u_x^{\infty})(u_x^n+u_x^{\infty})\|_{B^{s-1}_{p,r}}\\
&\leq C\|u_x^n-u_x^{\infty}\|_{B^{s-1}_{p,r}}.
\end{align*}
If $b=\frac{5}{3}$, then $F_3=0$. If $b\neq\frac{5}{3},\ s>\max\{\frac{5}{2},2+\frac{1}{p}\}$, then
\begin{align*}
\|\partial_x F_3(u^n)-\partial_x F_3(u^{\infty})\|_{B^{s-1}_{p,r}}\leq C\|(u_{xx}^n-u_{xx}^{\infty})(u_{xx}^n+u_{xx}^{\infty})\|_{B^{s-3}_{p,r}}
\leq C\|u_x^n-u_x^{\infty}\|_{B^{s-1}_{p,r}}.
\end{align*}
Therefore,
$$
\|z^n(t)\|_{B^{s-1}_{p,r}}
\leq e^{CMT}C\Big(\|u_{0}^n-u_{0}^\infty\|_{B^{s}_{p,r}}+\int_0^t \|u^n-u^{\infty}\|_{B^{s-1}_{p,r}}+\|u_x^n-u_x^{\infty}\|_{B^{s-1}_{p,r}}dt'\Big).
$$

Now we consider \eqref{wn}. Since $s>1+\frac{1}{p}$, we have
\begin{align*}
\|f^{\infty}\|_{B^{s-1}_{p,r}}\leq\|(\partial_x u^{\infty})^2\|_{B^{s-1}_{p,r}}+\|\partial_x F(u^{\infty})\|_{B^{s-1}_{p,r}}
\leq C\big(\|\partial_x u^{\infty}\|^2_{B^{s-1}_{p,r}}+\|u^{\infty}\|^2_{B^{s}_{p,r}}\big)\leq CM^2.    
\end{align*}
Recall that $\partial_x u_0^{\infty}\in B^{s-1}_{p,r}$, $\underset{n\in\overline{\mathbb{N}}}{\sup}\|u^n\|_{L^\infty_T(B^{s}_{p,r})}\leq M$, and $u^n\to u^\infty$ in $\mathcal{C}([0,T];B^{s-1}_{p,r})$. By Lemma \ref{existence}, we see that $w^n\in \mathcal{C}([0,T];B^{s-1}_{p,r})$. Then we deduce from Lemma \ref{zp19} that $w^n\to w^{\infty}$ in $\mathcal{C}([0,T];B^{s-1}_{p,r})$.

Hence, for all $\varepsilon>0$, there exits a constant $N\in\mathbb{N}$, such that $\forall\ n>N,\ t\in[0,T]$,
\begin{align*}
&\|w^n(t)-w^{\infty}(t)\|_{B^{s-1}_{p,r}}<\frac{\varepsilon}{4}e^{-CTe^{CMT}},\\
&\|u^n(t)-u^{\infty}(t)\|_{B^{s-1}_{p,r}}<\frac{\varepsilon}{4CTe^{CMT}}e^{-CTe^{CMT}}.
\end{align*}
Then we have
\begin{align*}
&\|\partial_x u^n(t)-\partial_x u^{\infty}(t)\|_{B^{s-1}_{p,r}}\\
\leq&\|z^n(t)\|_{B^{s-1}_{p,r}}+\|w^n(t)-w^{\infty}(t)\|_{B^{s-1}_{p,r}}\\
\leq&Ce^{CMT}\Big(\|u_{0}^n-u_{0}^\infty\|_{B^{s}_{p,r}}+\int_0^t \|\partial_x u^n(t')-\partial_x u^{\infty}(t')\|_{B^{s-1}_{p,r}}dt'\Big)+\frac{\varepsilon}{2}e^{-CTe^{CMT}}.
\end{align*}
By the Gronwall lemma, we can obtain
$$\|\partial_x u^n(t)-\partial_x u^{\infty}(t)\|_{B^{s-1}_{p,r}}
\leq\frac{\varepsilon}{2}+Ce^{CMT}e^{CTe^{CMT}}\|u_{0}^n-u_{0}^\infty\|_{B^{s}_{p,r}}.
$$
Since $u_0^n\to u_0^\infty$ in $B^{s}_{p,r}$, we have $\partial_x u^n\to\partial_x u^\infty$ in $\mathcal{C}([0,T];B^{s-1}_{p,r})$. Therefore,
$$\|u^n-u^{\infty}\|_{C([0,T];B^{s}_{p,r})}\leq C\big(\|\partial_x u^n-\partial_x u^{\infty}\|_{C([0,T];B^{s-1}_{p,r})}+\|u^n-u^{\infty}\|_{C([0,T];B^{s-1}_{p,r})}\big)\to 0$$
when $n\to\infty$, which also implies that $F(u^n)\to F(u^{\infty})$ in $\mathcal{C}([0,T];B^{s-1}_{p,r})$. Since
\begin{align*}
\|u^n u^n_x-u^{\infty}u^{\infty}_x\|_{B^{s-1}_{p,r}}
\leq C\big(\|u^n\|_{B^{s-1}_{p,r}}\|u^n-u^{\infty}\|_{B^{s}_{p,r}}+\|u^{\infty}\|_{B^{s}_{p,r}}\|u^n-u^{\infty}\|_{B^{s}_{p,r}}\big)
\to 0
\end{align*}
when $n\to\infty$, we have $u^n u^n_x\to u^{\infty}u^{\infty}_x$ in $\mathcal{C}([0,T];B^{s-1}_{p,r})$. Then by \eqref{eq1}, we see that $\partial_t u^n\to\partial_t u^\infty$ in $\mathcal{C}([0,T];B^{s-1}_{p,r})$. Hence, we conclude that $u^n\to u^{\infty}$ in $E^s_{p,r}(T)$.

\textbf{Case 2: $r=\infty$}

Similar to Case 1, we have $\underset{n\in\overline{\mathbb{N}}}{\sup}\|u^n\|_{L^\infty_T(B^{s}_{p,\infty})}\leq M$ and $\|u^n- u^\infty\|_{L^{\infty}([0,T];B^{s-1}_{p,\infty})}\to 0\ (n\to\infty)$. For all $\phi\in B^{-s}_{p',1}$, we see
\begin{align*}
\langle u^n(t)-u^{\infty}(t),\phi\rangle
&=\langle S_j(u^n(t)-u^{\infty}(t)),\phi\rangle+\langle(Id-S_j)(u^n(t)-u^{\infty}(t)),\phi\rangle\\
&=\langle u^n(t)-u^{\infty}(t),S_j\phi\rangle+\langle u^n(t)-u^{\infty}(t),(Id-S_j)\phi\rangle.
\end{align*}
Note that
$$
\underset{t\in[0,T]}{\sup}|\langle u^n(t)-u^{\infty}(t),S_j\phi\rangle|\leq C\|u^n- u^\infty\|_{L^{\infty}([0,T];B^{s-1}_{p,\infty})}\|S_j\phi\|_{B^{1-s}_{p',1}}
$$
and
\begin{align*}
\underset{t\in[0,T]}{\sup}|\langle u^n(t)-u^{\infty}(t),(Id-S_j)\phi\rangle|&\leq C\|u^n- u^\infty\|_{L^{\infty}([0,T];B^{s}_{p,\infty})}\|(Id-S_j)\phi\|_{B^{-s}_{p',1}}\\
&\leq 2CM\|(Id-S_j)\phi\|_{B^{-s}_{p',1}}\to 0\ (j\to\infty).
\end{align*}
For all $\varepsilon>0,$ we can fix $j\in\mathbb{N}$ large enough such that $\underset{t\in[0,T]}{\sup}|\langle u^n(t)-u^{\infty}(t),(Id-S_j)\phi\rangle|<\varepsilon/2$. For fixed $j$, there exists $N\in\mathbb{N}$ such that $\underset{t\in[0,T]}{\sup}|\langle u^n(t)-u^{\infty}(t),S_j\phi\rangle|<\varepsilon/2,\ \forall\ n>N$. Hence, we conclude that $u^n\to u^{\infty}$ in  $\mathcal{C}_w([0,T];B^s_{p,\infty})$.
\end{proof}

Next, taking advantage of the compactness method and the Lagrange coordinate transformation, we give the following local well-posedness result, which implies the index $\frac 52\ (resp.,\ \frac 32)$ in Theorem \ref{the1} is not necessary.
\begin{theo}\label{the2}
Let $u_{0}\in B_{p,1}^{s}$ with\ $1\leq p<\infty,$
$$
s=\left\{\begin{array}{l}
	2+\frac{1}{p} \quad \text{if} \quad b \neq \frac{5}{3}, \\
	1+\frac{1}{p} \quad \text{if} \quad b = \frac{5}{3}.
\end{array}\right.
$$
There exists a time $T>0$ such that\ \eqref{eq1} has a unique solution $u \in E_{p,1}^{s}(T)$. Furthermore, the solution map is continuous from $B^{s}_{p,1}$ to $E^{s}_{p,1}(T)$.
\end{theo}
\begin{proof}
We prove Theorem \ref{the2} in three steps.

\textbf{Step 1: Existence}

Similar to the proof of Theorem \ref{the1}, we set $u^0\triangleq0$ and define $(u^n)_{n\in\mathbb{N}}$ by solving equation \eqref{un}. We can similarly deduce that \eqref{un} has a solution $u^{n+1}\in E_{p,1}^{s}(T)$ for all positive $T$. Furthermore, for fixed $T>0$ such that $ 2C^2 T\|u_0\|_{B^{s}_{p,1}}<1$, we see that $(u^n)_{n\in \mathbb{N}}$ is uniformly bounded in $L^{\infty}([0,T];B^{s}_{p,1})$ and that $(\partial_t u^{n})_{n\in \mathbb{N}}$ is uniformly  bounded in $L^{\infty}([0,T];B^{s-1}_{p,1})$, which also implies that $(u^{n})_{n\in \mathbb{N}}$ is uniformly  bounded in $\mathcal{C}([0,T];B^{s}_{p,1})\cap \mathcal{C}^{\frac{1}{2}}([0,T];B^{s-1}_{p,1})$.

Now we use the compactness method for $(u^n)_{n\in\mathbb{N}}$ to get a solution $u$ of \eqref{eq1}. Let $(\varphi_j)_{j\in\mathbb{N}}$ be a sequence of smooth functions with value in $[0,1]$ supported in $B(0,j+1)$ and equal to 1 on $B(0,j)$. Then for all $j\in\mathbb{N},\ (\varphi_j u^{n})_{n\in \mathbb{N}}$ is uniformly  bounded in $\mathcal{C}([0,T];B^{s}_{p,1})\cap \mathcal{C}^{\frac{1}{2}}([0,T];B^{s-1}_{p,1})$. By Proposition \ref{prop1}, we see that $z\mapsto\varphi_j z$ is a compact operator from $B^s_{p,1}$ to $B^{s-1}_{p,1}$. Thus, we deduce from the Ascoli’s theorem and the Cantor’s diagonal process that there exists some function $u_j$ and a subsequence of $(u^n)_{n\in\mathbb{N}}$ (still denoted by $(u^n)_{n\in\mathbb{N}}$) such that $\varphi_j u^n\to u_j$ in $\mathcal{C}([0,T];B^{s-1}_{p,1}),\ \forall\ j\in\mathbb{N}$, which implies that there exists some function $u$ such that $\forall\ \varphi\in\mathcal{D},\ \varphi u^n\to\varphi u$ in $\mathcal{C}([0,T];B^{s-1}_{p,1})$. Due to the uniform boundedness of $(u^n)_{n\in \mathbb{N}}$ in $L^{\infty}([0,T];B^{s}_{p,1})$, the Fatou property guarantees that $u\in L^{\infty}([0,T];B^{s}_{p,1})$, so we have $\varphi u^n, \varphi u\in L^{\infty}([0,T];B^{s}_{p,1}),\ \forall\ \varphi\in\mathcal{D}$. Hence, an interpolation argument ensures that $\varphi u^n\to\varphi u$ in $\mathcal{C}([0,T];B^{s'}_{p,1}),\ \forall\ s'<s$. Therefore, it is a routine process to prove that $u$ is the solution of \eqref{eq1}. Then, similar to the proof of Theorem \ref{the1}, we have $u\in E^s_{p,1}(T)$.

\textbf{Step 2: Uniqueness}

If $b\neq\frac{5}{3}$, then $u\in \mathcal{C}([0,T];B_{p,1}^{2+\frac{1}{p}})\hookrightarrow \mathcal{C}([0,T];W^{2,\infty})$. If $b=\frac{5}{3}$, then $u\in \mathcal{C}([0,T];B_{p,1}^{1+\frac{1}{p}})\hookrightarrow \mathcal{C}([0,T];W^{1,\infty})$. Hence, the problem
\begin{equation}\label{y}
\left\{
\begin{array}{ll}
	\frac{d}{dt}y(t,\xi)=u\left(t,y(t,\xi)\right),&\quad t>0,\ \xi\in\mathbb{R}, \\
	y(0,\xi)=\xi,&\quad\xi\in\mathbb{R}
\end{array}\right.
\end{equation}
has a unique solution $y\in \mathcal{C}^1([0,T];\mathcal{C})$. Let $U(t,\xi)\triangleq u(t,y(t,\xi))$. We can obtain
\begin{align}
& U_\xi(t,\xi)=u_x(t,y(t,\xi))\ y_\xi(t,\xi),\label{Uxi}\\
& U_{\xi\xi}(t,\xi)=u_{xx}(t,y(t,\xi))\ y_\xi^2(t,\xi)+u_x(t,y(t,\xi))\ y_{\xi\xi}(t,\xi),\label{Uxixi}\\
& y(t,\xi)=\xi+\int_{0}^{t}U d\tau,\label{y0xi} \\
& y_{\xi}\left(t,\xi\right)=1+\int_{0}^{t}U_{\xi}d\tau,\label{yxi} \\
& U_{t}(t,\xi)=-F(u)\circ y\triangleq-\left(\widetilde{F}(U,y)\right),\label{Ut} \\
& U_{t\xi}(t,\xi)=-\left(\widetilde{F}(U,y)\right)_{\xi}.\label{Utxi}
\end{align}
Since
\begin{align*}
	y_{\xi}\left(t,\xi\right)&=\exp\left({\int_{0}^{t}(\partial_{\xi}u)(\tau,y(\tau,\xi))d\tau}\right)\in L^\infty\left([0,T]\times\mathbb{R}\right), \\
	y_{\xi\xi}\left(t,\xi\right)&=\exp\left({\int_{0}^{t}(\partial_{\xi}u)(\tau,y(\tau,\xi))d\tau}\right)\int_{0}^{t}(\partial_{\xi\xi}u)(\tau,y(\tau,\xi))\ y_{\xi}\left(\tau,\xi\right)d\tau\\
	&\in L^\infty\left([0,T]\times\mathbb{R}\right) \quad \text{if} \quad b\neq\frac{5}{3},
\end{align*}
we deduce from \eqref{Uxi} \eqref{Uxixi} that
\begin{align*}
U(t,\xi)\in\left\{
\begin{array}{ll}
L^\infty([0,T];W^{2,\infty}) \quad \text{if} \quad b\neq\frac{5}{3}, \\
L^\infty([0,T];W^{1,\infty}) \quad \text{if} \quad b=\frac{5}{3}.
\end{array}\right.
\end{align*}
Therefore, \eqref{yxi} ensures that 
\begin{equation} \label{bddyxi}
\frac{1}{2}\leq y_\xi\leq C_{u_0}
\end{equation}
for $T>0$ small enough, which implies the flow map
$y(t,\cdot)$ is an increasing diffeomorphism over $\mathbb{R}$. It follows that
\begin{align*}
 &\|U\|_{L^{p}}^{p}=\int_{\mathbb{R}}|U(t,\xi)|^{p}d\xi=\int_{\mathbb{R}}|u(t,y(t,\xi))|^{p}\frac{1}{y_{\xi}}dy\leq2\|u\|_{L^{p}}^{p}\leq C\|u\|_{B^{1+\frac 1p}_{p,1}}^{p}\leq C, \\
& \|U_{\xi}\|_{L^{p}}^{p}=\int_{\mathbb{R}}|U_{\xi}(t,\xi)|^{p}d\xi=\int_{\mathbb{R}}|u_{x}(t,y(t,\xi))|^{p}y_{\xi}^{p-1}dy\leq C_{u_{0}}^{p-1}\|u_{x}\|_{L^{p}}^{p}\leq C\|u\|_{B^{1+\frac 1p}_{p,1}}^{p}\leq C.
\end{align*}
If $b\neq\frac 53$, we have
\begin{align*} \|U_{\xi\xi}\|_{L^{p}}^{p}
=&\int_{\mathbb{R}}|u_{xx}(t,y(t,\xi))\ y_\xi^2(t,\xi)+u_x(t,y(t,\xi))\ y_{\xi\xi}(t,\xi)|^{p}d\xi\\
\leq &C_p\int_{\mathbb{R}}|u_{xx}(t,y(t,\xi))|^{p}\ y_\xi^{2p-1}(t,\xi)+|u_x(t,y(t,\xi))|^p|y_{\xi\xi}(t,\xi)|^p y_\xi^{-1}dy\\
\leq &C_p C_{u_{0}}^{2p-1}\|u_{xx}\|_{L^{p}}^{p}+C\|u_x\|_{L^{p}}^{p}\\
\leq &C\|u\|_{B^{2+\frac 1p}_{p,1}}^{p}\leq C.
\end{align*}
Together with \eqref{y0xi}, we gain
\begin{align}\label{bddU}
	U(t,\xi),\ y(t,\xi)-\xi\in\left\{
	\begin{array}{ll}
		L^\infty([0,T];W^{2,\infty}\cap W^{2,p}) \quad \text{if} \quad b\neq\frac{5}{3}, \\
		L^\infty([0,T];W^{1,\infty}\cap W^{1,p}) \quad \text{if} \quad b=\frac{5}{3}.
	\end{array}\right.
\end{align}

In order to prove the uniqueness, we suppose that $u_1, u_2$ are two solutions of \eqref{eq1} with the same initial data $u_0$. Then $U_i(t,\xi)\triangleq u_i(t,y_i(t,\xi))$ satisfies \eqref{Ut} and \eqref{Utxi}. We temporarily assume that the following lemma holds.

\begin{lemm}\label{F-F}
If $b\neq\frac 53$, then 
$$
\|\widetilde{F}(U_{1},y_{1})-\widetilde{F}(U_{2},y_{2})\|_{W^{2,\infty}\cap W^{2,p}}\leq C\left(\|U_1-U_2\|_{W^{2,\infty}\cap W^{2,p}}+\|y_1-y_2\|_{W^{2,\infty}\cap W^{2,p}}\right).
$$
Otherwise, if $b=\frac 53$, then 
$$
\|\widetilde{F}(U_{1},y_{1})-\widetilde{F}(U_{2},y_{2})\|_{W^{1,\infty}\cap W^{1,p}}\leq C\left(\|U_1-U_2\|_{W^{1,\infty}\cap W^{1,p}}+\|y_1-y_2\|_{W^{1,\infty}\cap W^{1,p}}\right).
$$
\end{lemm}
Taking advantage of Lemma \ref{F-F}, we see that if $b\neq\frac{5}{3}$, then
\begin{align*}
&\|U_1-U_2\|_{W^{2,\infty}\cap W^{2,p}}+\|y_1-y_2\|_{W^{2,\infty}\cap W^{2,p}}\\
\leq&C\Big(\|U_1(0)-U_2(0)\|_{W^{2,\infty}\cap W^{2,p}}+\|y_1(0)-y_2(0)\|_{W^{2,\infty}\cap W^{2,p}}\\
&+\int_{0}^{T}\|\widetilde{F}(U_{1},y_{1})-\widetilde{F}(U_{2},y_{2})\|_{W^{2,\infty}\cap W^{2,p}}+\|U_{1}-U_{2}\|_{W^{2,\infty}\cap W^{2,p}}dt\Big)\\
\leq&C\Big(\|U_1(0)-U_2(0)\|_{W^{2,\infty}\cap W^{2,p}}+\int_{0}^{T}\|U_1-U_2\|_{W^{2,\infty}\cap W^{2,p}}+\|y_1-y_2\|_{W^{2,\infty}\cap W^{2,p}}dt\Big).
\end{align*}
The Gronwall lemma gives
\begin{align*}
	&\|U_1-U_2\|_{W^{2,\infty}\cap W^{2,p}}+\|y_1-y_2\|_{W^{2,\infty}\cap W^{2,p}}\\
	\leq&C\|U_1(0)-U_2(0)\|_{W^{2,\infty}\cap W^{2,p}}\\
	\leq&C\|u_1(0)-u_2(0)\|_{B^{2+\frac 1p}_{p,1}}=0.
\end{align*}
Similarly, if $b=\frac{5}{3}$, we have 
$$
\|U_1-U_2\|_{W^{1,\infty}\cap W^{1,p}}+\|y_1-y_2\|_{W^{1,\infty}\cap W^{1,p}}\leq C\|u_1(0)-u_2(0)\|_{B^{1+\frac 1p}_{p,1}}=0.
$$
Hence, we see 
\begin{align*}
\|u_{1}-u_{2}\|_{B^0_{p,\infty}} & \leq\|u_{1}-u_{2}\|_{L^{p}} \leq C_{u_0}^{\frac 1p}\|u_{1}\circ y_{1}-u_{2}\circ y_{1}\|_{L^{p}} \\
& \leq C_{u_0}^{\frac 1p}\|u_{1}\circ y_{1}-u_{2}\circ y_{2}+u_{2}\circ y_{2}-u_{2}\circ y_{1}\|_{L^{p}} \\
& \leq C\|U_{1}-U_{2}\|_{L^{p}}+C\|u_{2x}\|_{L^{\infty}}\|y_{1}-y_{2}\|_{L^{p}}=0
\end{align*}
and obtain the uniqueness.

\textbf{Step 3: The continuous dependence}

Similar to the proof of Theorem \ref{the1}, we set $\overline{\mathbb{N}}=\mathbb{N}\cup\{\infty\}$ and fix $T>0$ such that $ 2C^2 T\underset{n\in\overline{\mathbb{N}}}{\sup}\|u_0^n\|_{B^{s}_{p,1}} <1$. Suppose that $u^n$ is the solution
of \eqref{eq1} with initial data $u_0^n\in B^{s}_{p,1}$ and that $u_0^n\to u_0^\infty$ in $B^{s}_{p,1}$ when $n\to\infty$. Then we have $\underset{n\in\overline{\mathbb{N}}}{\sup}\|u^n\|_{L^\infty_T(B^{s}_{p,r})}\leq M
$. By Step 2, we see that 
$$
\|(u^n-u^\infty)(t)\|_{B^{-1}_{p,1}}
\leq C\|(u^n-u^\infty)(t)\|_{L^p}\\
\leq C\|u_0^n-u_0^\infty\|_{B^{s}_{p,1}}
\to 0\ (n\to\infty).
$$
Taking advantage of the interpolation inequality, we have $u^n\to u^\infty$ in $\mathcal{C}([0,T];B^{s-1}_{p,1})$ when $n\to\infty$. Following the same line as Step 3 in the proof of Theorem \ref{the1}, we can obtain $\partial_x u^n\to\partial_x u^\infty$ in $\mathcal{C}([0,T];B^{s-1}_{p,1})$ and conclude that $u^n\to u^\infty$ in $E^{s}_{p,1}(T)$.
\end{proof}

For the sake of completeness, we now prove Lemma \ref{F-F}.
\begin{proof}[Proof of Lemma \ref{F-F}]
Without loss of generality, we assume that $\alpha=\beta=1$.

We mainly focus on the situation when $b\neq\frac 53$. The proof for the case when $b=\frac 53$ is similar and much easier, so we will omit it.

Denote $p(x)\triangleq\frac 12 e^{-|x|},\ G(x)\triangleq\frac 14 e^{-|x|}(1+|x|)$. It's easy to prove that 
\begin{align}
(1-\partial_{xx})^{-1}u=&p\ast u,\label{p}\\ (1-\partial_{xx})^{-1}(1-\partial_{xx})^{-1}u=&G\ast u.\label{G}
\end{align}
Let $P\triangleq \frac b2 u_{i}^{2}+\frac{1-b}{2}u_{ix}^{2}+\frac{5-3b}{2}u_{ixx}^{2}$, $Q\triangleq -\frac b2 u_i^2+2u_{ix}^{2}-\frac{5-3b}{2}u_{ixx}^{2}$. Then \eqref{eqF} gives
\begin{align}
	&\widetilde{F}(U_{i},y_{i})=F(u_i)\circ y_i\label{1}\\
	=&-\frac{1}{4}\int_{\mathbb{R}}\mathrm{sign}\left(y_{i}(t,\xi)-x\right)e^{-|y_{i}(t,\xi)-x|}|y_{i}(t,\xi)-x|Pdx\notag\\
	&+\frac{b-5}{4}\int_{\mathbb{R}}\mathrm{sign}\left(y_{i}(t,\xi)-x\right)e^{-|y_{i}(t,\xi)-x|}u_{ix}^{2}dx,\notag\\
&\left(\widetilde{F}(U_{i},y_{i})\right)_\xi= y_{i\xi}\cdot\partial_{x}F(u_i)\circ y_i\label{2}\\
=&\frac{b-5}{2}\frac{U_{i\xi}^2}{y_{i\xi}}+\frac{y_{i\xi}}{2}\int_{\mathbb{R}}e^{-|y_{i}(t,\xi)-x|}Qdx+\frac{y_{i\xi}}{4}\int_{\mathbb{R}}e^{-|y_{i}(t,\xi)-x|}(1+|y_{i}(t,\xi)-x|)Pdx,\notag\\
	&\left(\widetilde{F}(U_{i},y_{i})\right)_{\xi\xi}= y_{i\xi\xi}\cdot\partial_{x}F(u_i)\circ y_i+y_{i\xi}^2\cdot\partial_{xx}F(u_i)\circ y_i\label{3}\\
	=&\frac{b-5}{2}\frac{U_{i\xi}^2}{y_{i\xi}^2}\cdot y_{i\xi\xi}+\frac{y_{i\xi\xi}}{2}\int_{\mathbb{R}}e^{-|y_{i}(t,\xi)-x|}Qdx+\frac{y_{i\xi\xi}}{4}\int_{\mathbb{R}}e^{-|y_{i}(t,\xi)-x|}(1+|y_{i}(t,\xi)-x|)Pdx\notag\\
	&-y_{i\xi}^2\left(\frac{1}{4}\int_{\mathbb{R}}\mathrm{sign}\left(y_{i}(t,\xi)-x\right)e^{-|y_{i}(t,\xi)-x|}|y_{i}(t,\xi)-x|Pdx\right)\notag\\
	&+y_{i\xi}^2\left(\frac 12\int_{\mathbb{R}}\mathrm{sign}\left(y_{i}(t,\xi)-x\right)e^{-|y_{i}(t,\xi)-x|}Qdx\right)+(b-5)\ y_{i\xi}^2\ (u_{ix}u_{ixx})\circ y_i.\notag
\end{align}

We first estimate $\|\widetilde{F}(U_{1},y_{1})-\widetilde{F}(U_{2},y_{2})\|_{L^{p}\cap L^{\infty}}$. From \eqref{1}, we see that the most complicated term is
$$
\int_{\mathbb{R}}\mathrm{sign}\left(y_{1}(\xi)-x\right)e^{-|y_{1}(\xi)-x|}|y_{1}(\xi)-x|u_{1xx}^{2}dx-\int_{\mathbb{R}}\mathrm{sign}\left(y_{2}(\xi)-x\right)e^{-|y_{2}(\xi)-x|}|y_{2}(\xi)-x|u_{2xx}^{2}dx.
$$
Combining \eqref{Uxi} and \eqref{Uxixi}, we have
\begin{align*}
u_{ixx}(t,y_i(t,\xi))=&\left[U_{i\xi\xi}(t,\xi)-u_{ix}\left(t,y_i(t,\xi)\right)y_{i\xi\xi}(t,\xi)\right]\frac{1}{y_{i\xi}(t,\xi)^2}\\
=&\frac{U_{i\xi\xi}(t,\xi)}{y_{i\xi}(t,\xi)^2}-\frac{U_{i\xi}(t,\xi)y_{i\xi\xi}(t,\xi)}{y_{i\xi}(t,\xi)^3}\\
\triangleq& V_i(t,\xi).
\end{align*}
Since $y_i\ (i=1, 2)$ is monotonically increasing, we see that $\mathrm{sign}\left(y_{i}(\xi)-y_{i}(\eta)\right)=\mathrm{sign}(\xi-\eta)$. Hence, we get
\begin{align*}
&\int_{\mathbb{R}}\mathrm{sign}\left(y_{1}(\xi)-x\right)e^{-|y_{1}(\xi)-x|}|y_{1}(\xi)-x|u_{1xx}^{2}dx-\int_{\mathbb{R}}\mathrm{sign}\left(y_{2}(\xi)-x\right)e^{-|y_{2}(\xi)-x|}|y_{2}(\xi)-x|u_{2xx}^{2}dx\\
=&\int_{\mathbb{R}}\mathrm{sign}\left(y_{1}(\xi)-y_{1}(\eta)\right)e^{-|y_{1}(\xi)-y_{1}(\eta)|}|y_{1}(\xi)-y_{1}(\eta)|V_1(\eta)^2 y_{1\eta}d\eta\\
 &-\int_{\mathbb{R}}\mathrm{sign}\left(y_{2}(\xi)-y_{2}(\eta)\right)e^{-|y_{2}(\xi)-y_{2}(\eta)|}|y_{2}(\xi)-y_{2}(\eta)|V_2(\eta)^2 y_{2\eta}d\eta\\
=&\int_{\mathbb{R}}\mathrm{sign}(\xi-\eta)\left(e^{-|y_{1}(\xi)-y_{1}(\eta)|}|y_{1}(\xi)-y_{1}(\eta)|-e^{-|y_{2}(\xi)-y_{2}(\eta)|}|y_{2}(\xi)-y_{2}(\eta)|\right)V_1(\eta)^2 y_{1\eta}d\eta\\
&-\int_{\mathbb{R}}\mathrm{sign}(\xi-\eta)e^{-|y_{2}(\xi)-y_{2}(\eta)|}|y_{2}(\xi)-y_{2}(\eta)|\left(V_2(\eta)^2 y_{2\eta}-V_1(\eta)^2 y_{1\eta}\right)d\eta\\
\triangleq& I_1-I_2.
\end{align*}
Taking advantage of \eqref{y0xi}, we see that
\begin{align*}
I_1
=&\int_{-\infty}^{\xi}e^{-(\xi-\eta)}(\xi-\eta)\left(e^{-\int_0^t(U_1(\xi)-U_1(\eta))dt'}-e^{-\int_0^t(U_2(\xi)-U_2(\eta))dt'}\right)V_1(\eta)^2 y_{1\eta}d\eta\\
 &+\int_{-\infty}^{\xi}e^{-(\xi-\eta)}V_1(\eta)^2 y_{1\eta}\bigg(e^{-\int_0^t(U_1(\xi)-U_1(\eta))dt'}\int_0^t(U_1(\xi)-U_1(\eta))dt'\\
 &\quad\quad\quad\quad\quad\quad\quad\quad\quad\quad\quad\quad-e^{-\int_0^t(U_2(\xi)-U_2(\eta))dt'}\int_0^t(U_2(\xi)-U_2(\eta))dt'\bigg)d\eta\\
 &-\int^{+\infty}_{\xi}e^{\xi-\eta}(\eta-\xi)\left(e^{\int_0^t(U_1(\xi)-U_1(\eta))dt'}-e^{\int_0^t(U_2(\xi)-U_2(\eta))dt'}\right)V_1(\eta)^2 y_{1\eta}d\eta\\
 &-\int^{+\infty}_{\xi}e^{\xi-\eta}V_1(\eta)^2 y_{1\eta}\bigg(e^{\int_0^t(U_1(\xi)-U_1(\eta))dt'}\int_0^t(U_1(\eta)-U_1(\xi))dt'\\
 &\quad\quad\quad\quad\quad\quad\quad\quad\quad\quad\quad-e^{\int_0^t(U_2(\xi)-U_2(\eta))dt'}\int_0^t(U_2(\eta)-U_2(\xi))dt'\bigg)d\eta\\
 \leq&C\|U_1-U_2\|_{L^\infty}\bigg(\int_{\mathbb{R}}e^{-|\xi-\eta|}|\xi-\eta|V_1(\eta)^2 y_{1\eta}d\eta+\int_{\mathbb{R}}e^{-|\xi-\eta|}V_1(\eta)^2 y_{1\eta}d\eta\bigg).\\
\end{align*}
Combining \eqref{bddyxi} \eqref{bddU} and the Young's inequality, we gain
\begin{align*}
\|I_1\|_{L^\infty\cap L^p}\leq C\|U_1-U_2\|_{L^\infty}.
\end{align*}
Since \eqref{bddU} gives $|\int_0^t U_2(\xi)-U_2(\eta)dt'|\leq C$, we similarly have
\begin{align*}
I_2\leq C\bigg(\int_{\mathbb{R}}e^{-|\xi-\eta|}|\xi-\eta||V_2(\eta)^2 y_{2\eta}-V_1(\eta)^2 y_{1\eta}|d\eta+\int_{\mathbb{R}}e^{-|\xi-\eta|}|V_2(\eta)^2 y_{2\eta}-V_1(\eta)^2 y_{1\eta}|d\eta\bigg).
\end{align*}
By the definition of $V_i$, we see that 
\begin{align*}
\bigg|V_2(\eta)^2 y_{2\eta}-V_1(\eta)^2 y_{1\eta}\bigg|
\leq&\bigg|\frac{U_{1\eta\eta}^2 y_{1\eta}}{y_{1\eta}^4}-\frac{U_{2\eta\eta}^2 y_{2\eta}}{y_{2\eta}^4}\bigg|+\bigg|\frac{U_{1\eta}^2 y_{1\eta\eta}^2 y_{1\eta}}{y_{1\eta}^6}-\frac{U_{2\eta}^2 y_{2\eta\eta}^2 y_{2\eta}}{y_{2\eta}^6}\bigg|\\
&+2\bigg|\frac{U_{1\eta\eta}U_{1\eta} y_{1\eta\eta}y_{1\eta}}{y_{1\eta}^5}-\frac{U_{2\eta\eta}U_{2\eta} y_{2\eta\eta}y_{2\eta}}{y_{2\eta}^5}\bigg|.
\end{align*}
\eqref{bddyxi} and \eqref{bddU} ensure that
\begin{align}
&\bigg|\frac{U_{1\eta\eta}U_{1\eta} y_{1\eta\eta}y_{1\eta}}{y_{1\eta}^5}-\frac{U_{2\eta\eta}U_{2\eta} y_{2\eta\eta}y_{2\eta}}{y_{2\eta}^5}\bigg|\label{zuocha}\\
=&\bigg|\frac{(U_{1\eta\eta}U_{1\eta}-U_{2\eta\eta}U_{2\eta}) y_{1\eta\eta}y_{1\eta}y_{2\eta}^5+U_{2\eta\eta}U_{2\eta} (y_{1\eta\eta}y_{1\eta}y_{2\eta}^5-y_{2\eta\eta}y_{2\eta}y_{1\eta}^5)}{y_{1\eta}^5 y_{2\eta}^5}\bigg|\notag\\
\leq&C\left(|U_{1\eta\eta}U_{1\eta}-U_{2\eta\eta}U_{2\eta}|+|y_{1\eta\eta}y_{1\eta}y_{2\eta}^5-y_{2\eta\eta}y_{2\eta}y_{1\eta}^5|\right)\notag\\
\leq&C\left(|U_{1\eta\eta}(U_{1\eta}-U_{2\eta})|+|(U_{1\eta\eta}-U_{2\eta\eta})U_{2\eta}|+|y_{1\eta\eta}y_{1\eta}(y_{2\eta}^5-y_{1\eta}^5)|+|(y_{1\eta\eta}y_{1\eta}-y_{2\eta\eta}y_{2\eta})y_{1\eta}^5|\right)\notag\\
\leq&C\left(|U_{1\eta}-U_{2\eta}|+|U_{1\eta\eta}-U_{2\eta\eta}|+|y_{1\eta}-y_{2\eta}|+|y_{1\eta\eta}-y_{2\eta\eta}|\right).\notag
\end{align}
Calculating the other two terms in the same way, we eventually have
\begin{align*}
|V_2(\eta)^2 y_{2\eta}-V_1(\eta)^2 y_{1\eta}|
\leq C\left(|U_{1\eta}-U_{2\eta}|+|U_{1\eta\eta}-U_{2\eta\eta}|+|y_{1\eta}-y_{2\eta}|+|y_{1\eta\eta}-y_{2\eta\eta}|\right).
\end{align*}
Hence, the Young's inequality yields
\begin{align*}
\|I_2\|_{L^\infty\cap L^p}\leq C\left(\|U_1-U_2\|_{W^{2,\infty}\cap W^{2,p}}+\|y_1-y_2\|_{W^{2,\infty}\cap W^{2,p}}\right).
\end{align*}
Therefore, we complete the estimation for the most complicated term in $\|\widetilde{F}(U_{1},y_{1})-\widetilde{F}(U_{2},y_{2})\|_{L^{p}\cap L^{\infty}}$. Since \eqref{Uxi} gives
$ u_{ix}(t,y_i(t,\xi))=U_{i\xi}(t,\xi)/y_{i\xi}(t,\xi)$, we see from \eqref{1} that the estimation for other terms in $\|\widetilde{F}(U_{1},y_{1})-\widetilde{F}(U_{2},y_{2})\|_{L^{p}\cap L^{\infty}}$ is similar and much easier, so we omit the details and obtain
$$
\|\widetilde{F}(U_{1},y_{1})-\widetilde{F}(U_{2},y_{2})\|_{L^{p}\cap L^{\infty}}\leq C\left(\|U_1-U_2\|_{W^{2,\infty}\cap W^{2,p}}+\|y_1-y_2\|_{W^{2,\infty}\cap W^{2,p}}\right).
$$

As for $\|(\widetilde{F}(U_1,y_1))_\xi-(\widetilde{F}(U_2,y_2))_\xi\|_{L^\infty\cap L^p}$ and $\|(\widetilde{F}(U_1,y_1))_{\xi\xi}-(\widetilde{F}(U_2,y_2))_{\xi\xi}\|_{L^\infty\cap L^p}$, we see from \eqref{2} and \eqref{3} that the estimation of terms involving convolution is similar to what we mention above. Hence, we omit the details and conclude that 
\begin{align*}
\|\widetilde{F}(U_{1},y_{1})-\widetilde{F}(U_{2},y_{2})\|_{W^{2,\infty}\cap W^{2,p}}\leq C\left(\|U_1-U_2\|_{W^{2,\infty}\cap W^{2,p}}+\|y_1-y_2\|_{W^{2,\infty}\cap W^{2,p}}\right).
\end{align*}
\end{proof}

\section{Global existence}\label{Global existence}
In this section, we study the blow up criteria for \eqref{eq1} and conditions for global existence. Firstly, we give the following blow up criteria, which will be used in the proof of global existence and ill-posedness.
\begin{prop}\label{blow up criterion}
Let $1\leq p,r\leq\infty$ and $s$ satisfies \eqref{s}. Assume that $u$ is the solution of \eqref{eq1} with initial data $u_0\in B_{p, r}^s$. Define the lifespan $T^\ast_{u_0}$ of $u$ as the supremum of positive times $T$ such that \eqref{eq1} has a solution $u\in E^s_{p,r}(T)$ on $[0,T]\times\mathbb{R}$. If $T^\ast_{u_0}<\infty$, then we have
\begin{equation}\label{blow1}
	\left\{\begin{array}{ll}
	\int_0^{T^\ast_{u_0}}\|\partial_{x}u(\tau)\|_{L^{\infty}}+\|\partial_{xx}u(\tau)\|_{L^{\infty}}d\tau=\infty&\text{if}\quad b\neq\frac 53, \\
	\int_0^{T^\ast_{u_0}}\|\partial_{x}u(\tau)\|_{L^{\infty}}d\tau=\infty& \text{if}\quad b=\frac 53\\
\end{array}\right.
\end{equation}
and 
\begin{equation}\label{blow2}
	\left\{\begin{array}{ll}
		\int_0^{T^\ast_{u_0}}\|u(\tau)\|_{B^{2}_{\infty,\infty}}d\tau=\infty&\text{if}\quad b\neq\frac 53, \\
		\int_0^{T^\ast_{u_0}}\|u(\tau)\|_{B^{1}_{\infty,\infty}}d\tau=\infty& \text{if}\quad b=\frac 53.\\
	\end{array}\right.
\end{equation}
\end{prop}
\begin{proof}
We focus on the case when $b\neq\frac 53$.
Using Lemma \ref{priori estimate}, we can obtain
\begin{align*}
	\|u(t)\|_{B^{s}_{p,r}}\leq\left(\|u_{0}\|_{B^{s}_{p,r}}+\int_0^t  e^{-C\int_0^{t'}\|\partial_{x}u(\tau)\|_{L^\infty}d\tau}\|F(u)\|_{B^{s}_{p,r}}dt'\right)e^{C\int_0^t \|\partial_{x}u(\tau)\|_{L^\infty}d\tau}.
\end{align*}
Combining Proposition \ref{prop1}, Lemma \ref{product} and \eqref{eqF}, we have
$$
\|F(u)\|_{B^{s}_{p,r}}\leq C(\|u\|_{L^\infty}+\|u_x\|_{L^\infty}+\|u_{xx}\|_{L^\infty})\|u\|_{B^{s}_{p,r}}.
$$
Therefore, 
\begin{align*}
	e^{-C\int_0^t \|\partial_{x}u(\tau)\|_{L^\infty}d\tau} \|u(t)\|_{B^{s}_{p,r}}\leq\|u_{0}\|_{B^{s}_{p,r}}+C\int_0^t  e^{-C\int_0^{t'}\|\partial_{x}u(\tau)\|_{L^\infty}d\tau}\|u(t')\|_{B^{s}_{p,r}}\|u(t')\|_{W^{2,\infty}}dt'.
\end{align*}
According to the Gronwall lemma, we deduce that
\begin{align}
	\|u(t)\|_{B_{p,r}^{s}}\leq\|u_{0}\|_{B_{p,r}^{s}}\exp\left(C\int_{0}^{t}\|u(\tau)\|_{W^{2,\infty}}d\tau\right).\label{4.1}
\end{align}

Let $y$ be the solution of \eqref{y}. Then \eqref{eq1} gives 
$$
\partial_{t}(u(t,y(t,x)))=-F(u)(t,y(t,x)),
$$
which implies 
$$
\|u(t)\|_{L^\infty}\leq\|u_0\|_{L^\infty}+\int_0^t\|F(u)\|_{L^\infty}d\tau.
$$
Combining \eqref{p}, \eqref{G} and the Young's inequality, we see that
\begin{align*}
	&\|F_1(u)\|_{L^\infty}\leq C\|P(D)\partial_x(u^2)\|_{L^\infty}\leq C\|uu_x\|_{L^\infty}\leq C\|u\|_{L^\infty}\|u_x\|_{L^\infty},\\
	&\|F_2(u)\|_{L^\infty}\leq C\|\partial_x P(D)(u_x^2)\|_{L^\infty}\leq C\|u^2_x\|_{L^\infty}\leq C\|u_x\|^2_{L^\infty},\\
	&\|F_3(u)\|_{L^\infty}\leq C\|\partial_x P(D)(u_{xx}^2)\|_{L^\infty}\leq C\|u^2_{xx}\|_{L^\infty}\leq C\|u_{xx}\|^2_{L^\infty},\\
	&\|F_4(u)\|_{L^\infty}\leq C\|\partial_x^3 P(D)(u_x^2)\|_{L^\infty}\leq C\|u^2_x\|_{L^\infty}\leq C\|u_x\|^2_{L^\infty}.
\end{align*}
Therefore, we obtain
\begin{align}
	\|u(t)\|_{L^\infty}\leq\|u_0\|_{L^\infty}+C\int_0^t\left(\|\partial_{x}u\|_{L^{\infty}}+\|\partial_{xx}u\|_{L^{\infty}}\right)\|u\|_{W^{2,\infty}}d\tau.\label{a}
\end{align}
Similarly, by differentiating \eqref{eq1} once with respect to $x$, we have
\begin{align*}
	\left\{\begin{array}{l}
		\partial_{t}(u_x)+u(u_{x})_x+u_x^2+\partial_{x}F(u)=0,\\
		u_x|_{t=0}=\partial_{x}u_{0},
	\end{array}\right.
\end{align*}
which implies
$$
\partial_{t}(u_x(t,y(t,x)))=-(u_x^2+\partial_{x}F(u))(t,y(t,x))
$$
and gives
\begin{align}
	\|\partial_{x}u(t)\|_{L^\infty}
	\leq\|\partial_{x}u_0\|_{L^\infty}+C\int_0^t\left(\|\partial_{x}u\|_{L^{\infty}}+\|\partial_{xx}u\|_{L^{\infty}}\right)\|u\|_{W^{2,\infty}}d\tau.\label{b}
\end{align}
Differentiating \eqref{eq1} twice with respect to $x$ yields
\begin{align*}
	\left\{\begin{array}{l}
		\partial_{t}(u_{xx})+u(u_{xx})_x+3u_x u_{xx}+\partial_{xx}F(u)=0,\\
		u_{xx}|_{t=0}=\partial_{xx}u_{0},
	\end{array}\right.
\end{align*}
so we similarly deduce that
\begin{align}
	\|\partial_{xx}u(t)\|_{L^\infty}
	\leq\|\partial_{xx}u_0\|_{L^\infty}+C\int_0^t\left(\|\partial_{x}u\|_{L^{\infty}}+\|\partial_{xx}u\|_{L^{\infty}}\right)\|u\|_{W^{2,\infty}}d\tau.\label{c}
\end{align}
Combining \eqref{a}, \eqref{b} and \eqref{c}, we obtain
\begin{align}
\|u(t)\|_{W^{2,\infty}}\leq\|u_{0}\|_{W^{2,\infty}}\exp\left(C\int_{0}^{t}\|\partial_{x}u(\tau)\|_{L^{\infty}}+\|\partial_{xx}u(\tau)\|_{L^{\infty}}d\tau\right).\label{4.2}
\end{align}

Now we prove \eqref{blow1} when $b\neq\frac 53$. From Theorem \ref{the1}, we see that there exists a $T^\ast>0$ such that \eqref{eq1} has a unique solution $u\in\underset{T<T^\ast}{\bigcap}E^s_{p,r}(T)$. Then we have $T^\ast\leq T^\ast_{u_0}<\infty$. We want to show that if
$$
\int_0^{T^\ast}\|\partial_{x}u(\tau)\|_{L^{\infty}}+\|\partial_{xx}u(\tau)\|_{L^{\infty}}d\tau<\infty,
$$
then $T^\ast<T^\ast_{u_0}$.
Indeed, taking advantage of \eqref{4.1} and \eqref{4.2}, we see that there exists a constant $M_{T^\ast}>0$ such that for all $t\in[0,T^\ast),\ \|u(t)\|_{B^s_{p,r}}\leq M_{T^\ast}$. Let $\varepsilon>0$ be such that $2C^2M_{T^\ast}\varepsilon<1$, where $C$ is the constant in the proof of Theorem \ref{the1}. We then have a solution $\widetilde{u}\in E^s_{p,r}(\varepsilon)$ of \eqref{eq1} with initial data $u(T^\ast-\varepsilon/2)$. By uniqueness, we have $\widetilde{u}(t)=u(t+T^{\ast}-\varepsilon/2)$ on $[0,\varepsilon/2)$, which means that $\widetilde{u}$ extends the solution $u$ beyond $T^\ast$. Hence, we obtain $T^\ast<T^\ast_{u_0}$ and conclude that \eqref{blow1} holds when $b\neq\frac 53$.

Next, we prove \eqref{blow2} when $b\neq\frac 53$. Thanks to 
$$
x\left(1+\log\frac{x_0}{x}\right)\leq 1+x\log(e+x_0),\ \forall\ 0\leq x\leq 2x_0,
$$
we deduce from Lemma \ref{log} and Proposition \ref{prop1} that
\begin{align*}
\|u\|_{W^{2,\infty}}
\leq C\bigg(1+\|u\|_{B^2_{\infty,\infty}}\log\left(e+\|u\|_{B^s_{p,r}}\right)\bigg),
\end{align*}
which together with \eqref{4.1} gives 
\begin{align*}
\|u(t)\|_{B_{p,r}^{s}}\leq\|u_{0}\|_{B_{p,r}^{s}}\exp\left(C\int_{0}^{t}1+\|u(\tau)\|_{B^2_{\infty,\infty}}\log\left(e+\|u(\tau)\|_{B^s_{p,r}}\right)d\tau\right).
\end{align*}
Therefore, we have
\begin{align*}
\log\left(e+\|u(t)\|_{B_{p,r}^s}\right)\leq\log\left(e+\|u_0\|_{B_{p,r}^s}\right)+Ct+C\int_{0}^{t}\|u\|_{B_{\infty,\infty}^2}\log\left(e+\|u\|_{B_{p,r}^s}\right)d\tau.
\end{align*}
The Gronwall lemma thus yields
\begin{align*}
\log\left(e+\|u(t)\|_{B_{p,r}^{s}}\right)\leq\left(\log\left(e+\|u_{0}\|_{B_{p,r}^{s}}\right)+Ct\right)e^{C\int_{0}^{t}\|u\|_{B_{\infty,\infty}^{2}}d\tau}.
\end{align*}
If $\int_0^{T^\ast_{u_0}}\|u(\tau)\|_{B^{2}_{\infty,\infty}}d\tau<\infty$, then we similarly deduce that $u$ can be extended beyond $T^\ast_{u_0}$, which stands in contradiction to the definition of $T^\ast_{u_0}$. Hence, we conclude that \eqref{blow2} holds when $b\neq\frac 53$.

Following the same line, we can prove that when $b=\frac 53$,
\begin{align}
	&\|u(t)\|_{B_{p,r}^{s}}\leq\|u_{0}\|_{B_{p,r}^{s}}\exp\left(C\int_{0}^{t}\|u(\tau)\|_{W^{1,\infty}}d\tau\right),\label{4.3}\\
	&\|u(t)\|_{W^{1,\infty}}\leq\|u_{0}\|_{W^{1,\infty}}\exp\left(C\int_{0}^{t}\|\partial_{x}u(\tau)\|_{L^{\infty}}d\tau\right).\label{4.4}
\end{align}
Taking advantage of \eqref{4.3} and \eqref{4.4}, we can prove that \eqref{blow1} and \eqref{blow2} hold when $b=\frac 53$ in the same way, so we omit the details here. 
\end{proof}

\begin{theo}\label{global}
Suppose $u_0\in H^s(\mathbb{R})\cap W^{4,\frac 1b}(\mathbb{R})$ with $0\leq b\leq 1,\ s>\frac 52$. Then the corresponding solution $u$ of \eqref{eq1} exists globally in time.
\end{theo}
\begin{proof}
Let $y$ be the solution of \eqref{y} and $T>0$ be the maximal existence time of $u$. Taking advantage of \eqref{eq0} and \eqref{y}, we have
\begin{align*}
&\frac{\partial}{\partial t}\left(m(t,y(t,x))\cdot(\partial_{x}y(t,x))^b\right)\\
=&(\partial_t m)(t,y(t,x))\cdot(\partial_{x}y(t,x))^b+(\partial_x m)(t,y(t,x))\cdot\partial_{t}y(t,x)\cdot(\partial_{x}y(t,x))^b\\
&+m(t,y(t,x))\cdot b(\partial_{x}y(t,x))^{b-1}\partial_{t}\partial_{x}y(t,x)\\
=&(\partial_{x}y(t,x))^b\cdot(\partial_{t}m+\partial_{x}m\cdot u+bm\partial_{x}u)(t, y(t,x))\\
=&0,
\end{align*}
which implies
$$
m(t,y(t,x))\cdot(\partial_{x}y(t,x))^b=m_0(x).
$$
Thanks to Step 2 in the proof of Theorem \ref{the2}, we see that $\partial_{x}y\left(t,x\right)=\exp\left({\int_{0}^{t}(\partial_{x}u)(\tau,y(\tau,x))d\tau}\right)>0$ and that $y(t,\cdot)$ is an increasing diffeomorphism. Consequently, if $0<b\leq1$, we can obtain
\begin{align*}
\|m(t,\cdot)\|_{L^{\frac 1b}}^{\frac 1b}=&\int_{\mathbb{R}}|m(t,y(t,x))|^{\frac 1b}dy(t,x)\\
=&\int_{\mathbb{R}}|m(t,y(t,x))|^{\frac 1b}\partial_{x}y(t,x)dx\\
=&\int_{\mathbb{R}}|m_0(x)|^{\frac 1b}dx\\
=&\|m_0\|_{L^{\frac 1b}}^{\frac 1b}.
\end{align*}
If $b=0$, it's obvious that $\|m(t,\cdot)\|_{L^{\infty}}=\|m_0\|_{L^{\infty}}$. Since $m=(1-\alpha^2\partial_{xx})(1-\beta^2\partial_{xx})u$, by the Sobolev imbedding theorem, we see that for all $t\in [0,T)$, 
\begin{align*}
\|u(t,\cdot)\|_{W^{2,\infty}}\leq C\|u(t,\cdot)\|_{W^{4,\frac 1b}}\leq C\|m(t,\cdot)\|_{L^{\frac 1b}}=C\|m_0\|_{L^{\frac 1b}}\leq C\|u_0\|_{W^{4,\frac 1b}}.
\end{align*}
Hence, by Proposition \ref{blow up criterion}, we conclude that $u$ exists globally in time.
\end{proof}

\section{Ill-posedness}\label{Ill-posedness}
In this section, we provide two ill-posedness results when $b=\frac 53$. Firstly, we prove that the Cauchy problem \eqref{eq1} is ill-posed in $B^1_{\infty,1}$ in the sense of norm inflation. Let $B^0_{\infty,\infty,1}$ be a Banach space equipped with the norm $\|f\|_{B^0_{\infty,\infty,1}}=\underset{j\geq-1}{\sup}(j+2)^{1+\frac{1}{100}}\|\Delta_j f\|_{L^\infty}$. We have the following lemmas.

\begin{lemm}\label{illlm1}\cite{DAIbushiding,GUObushiding}
For any $f\in B^0_{\infty,1}\cap B^0_{\infty,\infty,1},\ g\in B^0_{\infty,1}$, we have
\begin{align*}
&\|f^{2}\|_{B_{\infty,\infty,1}^{0}}\leq C\|f\|_{B_{\infty,1}^{0}}\|f\|_{B_{\infty,\infty,1}^{0}},\\
&\|f^{2}\|_{B_{\infty,1}^{0}}\leq C\|f\|_{B_{\infty,1}^{0}}\|f\|_{B_{\infty,\infty,1}^{0}},\\
&\|fg\|_{B_{\infty,1}^{0}}\leq C\|f\|_{B_{\infty,1}^{0}\cap B_{\infty,\infty,1}^{0}}\|g\|_{B_{\infty,1}^{0}}.
\end{align*}
\end{lemm}

\begin{lemm}\label{illlm2}
Denote $R_{j}\triangleq f\Delta_{j}g_{x}-\Delta_{j}(fg_{x})=[f,\Delta_{j}]\partial_{x}g$, then we have
\begin{align*}
&\sum_j2^j\|R_j\|_{L^\infty}\leq C\|f_x\|_{B_{\infty,1}^0}\|g\|_{B_{\infty,1}^1},\\
&\underset{j\geq-1}{\sup}\left((j+2)^{1+\frac{1}{100}}\|R_j\|_{L^\infty} \right)\leq C\left(\|f_x\|_{B_{\infty,1}^0}\|g\|_{B_{\infty,\infty,1}^0}+\|f_x\|_{B_{\infty,\infty,1}^0}\|g\|_{B_{\infty,1}^0}\right).
\end{align*}
\end{lemm}
\begin{proof}
The first inequality can be deduced from Lemma 2.100 in \cite{BCD}, so we only prove the second inequality.

In order to prove the second inequality, we split $f$ into low and high frequencies: $f=S_0 f+\widetilde{f}$. As $\widetilde{f}$ is spectrally supported away from the origin, Proposition \ref{Bernstein’s inequalities} ensures that
\begin{align}\label{high}
\forall\ a\in[1,\infty],\ \forall\ j\geq-1,\ \|\Delta_j\partial_{x}\widetilde{f}\|_{L^a}\approx2^j\|\Delta_j\widetilde{f}\|_{L^a}.
\end{align}
Applying the Bony decomposition, we have $R_j=\sum_{i=1}^8 R_j^i$, where
\begin{align*}
& R_j^1=[T_{\widetilde{f}},\Delta_{j}]\partial_{x}g, & & R_j^{2} =T_{\partial_{x}\Delta_{j}g}\widetilde{f}, \\
& R_j^3=-\Delta_jT_{\partial_x g}\widetilde{f}, & & R_j^4 =\partial_x R(\widetilde{f},\Delta_j g), \\
& R_j^5=-R(\partial_{x}\widetilde{f},\Delta_{j}g), & & R_j^6 =-\partial_{x}\Delta_{j}R(\widetilde{f},g), \\
& R_j^7=\Delta_{j} R(\partial_{x}\widetilde{f},g), & & R_j^8 =[S_0f,\Delta_j]\partial_x g.
\end{align*}
Since $R_j^1=\sum_{|j-j'|\leq4}[S_{j^{\prime}-1}\widetilde{f},\Delta_j]\partial_x\Delta_{j'}g$, we deduce from Lemma \ref{commutator} that for all $j\geq-1$, 
\begin{align*}
(j+2)^{1+\frac{1}{100}}\|R_j^1\|_{L^\infty}\leq &C(j+2)^{1+\frac{1}{100}}\sum_{|j-j'|\leq4}2^{-j}\|\partial_{x}S_{j^{\prime}-1}\widetilde{f}\|_{L^\infty}\|\partial_x\Delta_{j'}g\|_{L^\infty}\\
\leq&C(j+2)^{1+\frac{1}{100}}\sum_{|j-j'|\leq4}\|f_x\|_{B_{\infty,1}^0}\|\Delta_{j'}g\|_{L^\infty}\\
\leq&C\|f_x\|_{B_{\infty,1}^0}\|g\|_{B_{\infty,\infty,1}^0}.
\end{align*}
Taking advantage of Proposition \ref{Bernstein’s inequalities}, we see that for all $j\geq-1$, 
\begin{align*}
	(j+2)^{1+\frac{1}{100}}\|R_j^2\|_{L^\infty}\leq &C(j+2)^{1+\frac{1}{100}}
	\sum_{j'\geq j+1}1_{j'\geq1}
	\|S_{j^{\prime}-1}\partial_{x}\Delta_{j}g\|_{L^\infty}\|\Delta_{j'}\widetilde{f}\|_{L^\infty}\\
	\leq&C
	\sum_{j'\geq j+1}1_{j'\geq1}
	(j+2)^{1+\frac{1}{100}}\|\Delta_{j}g\|_{L^\infty}\|\Delta_{j'}f_x\|_{L^\infty}\\
	\leq&C\|f_x\|_{B_{\infty,1}^0}\|g\|_{B_{\infty,\infty,1}^0},\\
	(j+2)^{1+\frac{1}{100}}\|R_j^3\|_{L^\infty}\leq &C(j+2)^{1+\frac{1}{100}}
	\sum_{|j'-j|\leq4}1_{j'\geq1}\sum_{j''\leq j'-2}
	\|\Delta_{j}(\Delta_{j''}\partial_{x}g\Delta_{j'}\widetilde{f})\|_{L^\infty}\\
	\leq&C
	\sum_{|j'-j|\leq4}1_{j'\geq1}(j'+2)^{1+\frac{1}{100}}\sum_{j''\leq j'-2}
	\|\Delta_{j''}g\|_{L^\infty}\|\Delta_{j'}f_x\|_{L^\infty}\\
	\leq&C\|f_x\|_{B_{\infty,\infty,1}^0}\|g\|_{B_{\infty,1}^0}.
\end{align*}
Denote $\widetilde{\Delta}_{j'}\triangleq\Delta_{j^{\prime}-1}+\Delta_{j^{\prime}}+\Delta_{j^{\prime}+1}$. Using Proposition \ref{Bernstein’s inequalities} and \eqref{high}, we infer that for all $j\geq-1$,
\begin{align*}
	(j+2)^{1+\frac{1}{100}}\|R_j^4\|_{L^\infty}\leq &C(j+2)^{1+\frac{1}{100}}
	\sum_{|j'-j|\leq2}\left(\|\partial_{x}\Delta_{j'}\widetilde{f}\|_{L^\infty}\|\Delta_{j}\widetilde{\Delta}_{j'}g\|_{L^\infty}+\|\Delta_{j'}\widetilde{f}\|_{L^\infty}\|\partial_{x}\Delta_{j}\widetilde{\Delta}_{j'}g\|_{L^\infty}\right)\\
	\leq &C(j+2)^{1+\frac{1}{100}}
	\sum_{|j'-j|\leq2}\|\Delta_{j'}\widetilde{f}_x\|_{L^\infty}\|\Delta_{j}g\|_{L^\infty}\\
	\leq&C\|f_x\|_{B_{\infty,1}^0}\|g\|_{B_{\infty,\infty,1}^0}.
\end{align*}
Similarly, we can prove that for $5\leq i\leq8$,
$$
\underset{j\geq-1}{\sup}\left((j+2)^{1+\frac{1}{100}}\|R_j^i\|_{L^\infty} \right)\leq C\|f_x\|_{B_{\infty,1}^0}\|g\|_{B_{\infty,\infty,1}^0},
$$
so we omit the details and complete the proof.
\end{proof}

Taking advantage of the two lemmas above, we now give the ill-posedness in $B^1_{\infty,1}$.
\begin{theo}\label{ill1}
Let $b=\frac 53$. Then for any $10<N\in\mathbb{N}^+$ large enough, there exists a $u_0\in\mathcal{C}^\infty(\mathbb{R})$ such that the following hold:\\
(i) $\|u_{0}\|_{B^{1}_{\infty,1}} \leq C N^{-\frac{1}{10}}$;\\
(ii) There exists a unique solution $u \in \mathcal{C}_{T}\left(\mathcal{C}^{\infty}(\mathbb{R})\right)$ to the Cauchy problem \eqref{eq1} with a time $T \leq 2N^{-\frac{1}{2}}$;\\
(iii) There exists a time $t_{0} \in [0, T]$ such that $\|u(t_{0})\|_{B^{1}_{\infty,1}} \geq \ln N$.
\end{theo}
\begin{proof}
Let 
\begin{align}
u_0(x)=-(1-\partial_{xx})^{-1}\partial_x\left[\cos2^{N+5}x\cdot\left(1+N^{-\frac{1}{10}}S_Nh(x)\right)\right]N^{-\frac{1}{10}},
\end{align}
where $S_{N}f=\underset{-1\leq j<N}{\sum}\Delta_{j}f$ and $h(x)=1_{x\geq0}(x)$. Similar to the proof of Lemma 4.1 in \cite{GUObushiding}, we have
\begin{align}
&\|u_{0}\|_{B^{1}_{\infty,1}} \leq C N^{-\frac{1}{10}},\label{uoB1.1}\\
&\|u_{0x}\|_{B^{0}_{\infty,\infty,1}} \leq C N^{\frac{9}{10}+\frac{1}{100}},\label{uoxBo..1}\\
&\|u_{0x}^2\|_{B^{0}_{\infty,1}} \geq C' N^{\frac{3}{5}}.\label{uox2Bo.1}
\end{align}
Owing to the proof of Proposition \ref{blow up criterion}, we have
\begin{align*}
\|u(t)\|_{W^{1,\infty}}\leq \|u_0\|_{W^{1,\infty}}+\int_0^t C_{\alpha,\beta}\|u(\tau)\|^2_{W^{1,\infty}}d\tau.
\end{align*}
Denote $T_0\triangleq\frac{1}{4C_{\alpha,\beta}\|u_0\|_{W^{1,\infty}}}$. For all $t\in[0,T_0]$, we get
\begin{align}\label{u(t)W1}
\|u(t)\|_{W^{1,\infty}}\leq 2\|u_0\|_{W^{1,\infty}}\leq C\|u_0\|_{B^{1}_{\infty,1}}\leq C N^{-\frac{1}{10}}.
\end{align}
Let $y$ be the solution of \eqref{y}. There is a time $T_1>0$ small enough such that \eqref{bddyxi} holds for any $t\in[0,\min\{T_0,T_1\}]$. Let $\bar{T}\triangleq2N^{-\frac 12}\leq\min\{T_0,T_1\}$ for $N>10$ sufficiently large. In order to prove the norm inflation, it suffices to prove that there is a time $t_0\in[0,\bar{T}]$ such that $\|u(t_0)\|_{B^1_{\infty,1}}\geq \ln N$. We assume the opposite holds, that is,
\begin{align}\label{opposite}
\sup_{t\in[0,\bar{T}]}\left\|u(t)\right\|_{B_{\infty,1}^{1}}\leq\ln N.
\end{align}
We see from \eqref{eq1} that
\begin{align*}
\partial_{t}\left(\Delta_{j}u\right)+\Delta_{j}\left(uu_x\right)+\Delta_{j}F(u)=0.
\end{align*}
Applying the Lagrange coordinates to the above equation and integrating with respect to $t$, we obtain
\begin{align}\label{Deltaju}
\left(\Delta_{j}u\right)\circ y=\Delta_{j}u_0+\int_0^t\sum_{i=1}^4 I_i(\tau)d\tau-\Delta_j F_4(u_0)\cdot t,
\end{align}
where
\begin{align*}
&I_1=(u\Delta_ju_x-\Delta_j(uu_x))\circ y\triangleq R_j\circ y,\\
&I_2=-(\Delta_j F_1(u))\circ y,\\
&I_3=-(\Delta_j F_2(u))\circ y,\\
&I_4=\Delta_j F_4(u_0)-(\Delta_j F_4(u))\circ y.
\end{align*}
Let $T=N^{-\frac 12}$. Since we want an estimate of $\|u\|_{B^1_{\infty,1}}$, we need to calculate $\underset{j\geq-1}{\sum}2^j\|I_i\|_{L^\infty}$. Taking advantage of \eqref{u(t)W1} and \eqref{opposite}, we have the following estimates:\\
(1) Similar to the proof of Lemma 2.100 in \cite{BCD}, we have
$$\underset{j\geq-1}{\sum}2^j\|I_1\|_{L^\infty}\leq\underset{j\geq-1}{\sum}2^j\|R_j\|_{L^\infty}\leq C\|u_x\|_{L^\infty}\|u\|_{B_{\infty,1}^1}\leq CN^{-\frac{1}{10}}\ln N.$$
(2) Thanks to Proposition \ref{prop1} and Lemma \ref{product}, we get
\begin{align*}
&\sum_{j\geq-1}2^j\|I_2\|_{L^\infty}\leq\sum_{j\geq-1}2^j\|\Delta_j F_1(u)\|_{L^\infty}\leq C\|u^2\|_{B_{\infty,1}^{-2}}\leq C\|u^2\|_{L^\infty}\leq C\|u\|_{L^\infty}^2 \leq CN^{-\frac{1}{5}},\\
&\sum_{j\geq-1}2^j\|I_3\|_{L^\infty}\leq\sum_{j\geq-1}2^j\|\Delta_j F_2(u)\|_{L^\infty}\leq C\|u_x^2\|_{B_{\infty,1}^{-2}}\leq C\|u_x^2\|_{L^\infty}\leq C\|u_x\|_{L^\infty}^2\leq CN^{-\frac{1}{5}}.
\end{align*}
(3) Now we consider $\underset{j\geq-1}{\sum}2^j\|I_4\|_{L^\infty}$, which is the most challenging to estimate and demonstrates the significance of Lemma \ref{illlm1} and Lemma \ref{illlm2}. Denote $E=\partial_{x}^3P(D)(u_x^2)$, then \eqref{eq1} gives
\begin{equation}\label{eqE}
\left\{
\begin{array}{l}
	\partial_t E+u\partial_xE=G(t,x),\quad t\in[0,T], \\
	E|_{t=0}=E_0=\partial_{x}^3P(D)(u_{0x}^2),
\end{array}\right.
\end{equation}
where
\begin{align*}
&G(t,x)=-2\partial_{x}^3P(D)(u_x\partial_{x} F(u))-\partial_{x}^3P(D)(u_x^3)+V_1+P(D)V_2,\\
&V_1=\frac{1}{\alpha^4}u\cdot P(D)(u_x^2)-\frac{1}{\alpha^2}\left(\frac{1}{\alpha^2}+\frac{1}{\beta^2}\right)u\cdot\left(1-\beta^2\partial_{xx}\right)^{-1}(u_x^2),\\
&V_2=\frac{1}{\alpha^2\beta^2}uu_x^2-\left(\frac{1}{\alpha^2}+\frac{1}{\beta^2}\right)\partial_{xx}(uu_x^2).
\end{align*}
Taking advantage of Lemma \ref{illlm1}, \eqref{u(t)W1} and \eqref{opposite}, we have
\begin{align*}
&\|\partial_{x}^3P(D)(u_x\partial_{x} F(u))\|_{B^1_{\infty,1}}\\
\leq& C\|u_x\partial_{x} F(u)\|_{B^0_{\infty,1}}\\
\leq& C\left(\|u_x\partial_{xx}P(D)u^2\|_{B^0_{\infty,1}}+\|u_x\partial_{xx}P(D)u_x^2\|_{B^0_{\infty,1}}+\|u_x\partial_{x}^4 P(D)u_x^2\|_{B^0_{\infty,1}}\right)\\
\leq& C\|u_x\|_{B_{\infty,1}^{0}\cap B_{\infty,\infty,1}^{0}}\left(\|\partial_{xx}P(D)u^2\|_{B^0_{\infty,1}}+\|\partial_{xx}P(D)u_x^2\|_{B^0_{\infty,1}}\right)\\
&+C\|u_x\|_{B_{\infty,1}^{0}}\|\partial_{x}^4 P(D)u_x^2\|_{B^0_{\infty,1}\cap B_{\infty,\infty,1}^{0}}\\
\leq& C\|u_x\|_{B_{\infty,1}^{0}\cap B_{\infty,\infty,1}^{0}}\|u\|_{W^{1,\infty}}^2+C\|u_x\|_{B_{\infty,1}^{0}}\|u_x\|_{B^0_{\infty,1}}\|u_x\|_{B^0_{\infty,\infty,1}}\\
\leq &C\left(N^{-\frac 15}\ln N+N^{-\frac 15}\|u_x\|_{B^0_{\infty,\infty,1}}+(\ln N)^2\|u_x\|_{B^0_{\infty,\infty,1}}\right)
\end{align*}
and
\begin{align*}
\|\partial_{x}^3P(D)(u_x^3)\|_{B^1_{\infty,1}}\leq &C\|u_x^3\|_{B^0_{\infty,1}}\leq C\|u_x^2\|_{B_{\infty,1}^{0}\cap B_{\infty,\infty,1}^{0}}\|u_x\|_{B_{\infty,1}^{0}}\\
\leq& C\|u_x\|_{B_{\infty,\infty,1}^{0}}\|u_x\|_{B_{\infty,1}^{0}}^2\\
\leq& C(\ln N)^2\|u_x\|_{B^0_{\infty,\infty,1}}.
\end{align*}
Moreover, it's easy to check that
\begin{align*}
\|V_1+P(D)V_2\|_{B^1_{\infty,1}}\leq C\|u\|_{W^{1,\infty}}^2\|u\|_{B^1_{\infty,1}}\leq CN^{-\frac 15}\ln N.
\end{align*}
Since $N>10$ , we obtain
\begin{align*}
\|G(t)\|_{B^1_{\infty,1}}
\leq C\left(N^{-\frac 15}\ln N+(\ln N)^2\|u_x(t)\|_{B^0_{\infty,\infty,1}}\right).
\end{align*}
Applying $\Delta_{j}$ and the Lagrange coordinates to \eqref{eqE} yields
\begin{align*}
\left(\Delta_{j}E\right)\circ y-\Delta_{j}E_0=\int_0^t\widetilde{R_j}\circ y+\left(\Delta_{j}G\right)\circ y\ d\tau,
\end{align*}
where $\widetilde{R_j}=u\partial_{x}\Delta_{j}E-\Delta_{j}(u\partial_{x}E)$. According to Lemma \ref{illlm1}, Lemma \ref{illlm2} and \eqref{opposite}, we have
\begin{align*}
\sum_{j\geq-1}2^j\|\widetilde{R_j}\circ y\|_{L^\infty}\leq&\sum_{j\geq-1}2^j\|\widetilde{R_j}\|_{L^\infty}\leq C\|u_x\|_{B_{\infty,1}^{0}}\|E\|_{B_{\infty,1}^{1}}\leq C\|u\|_{B_{\infty,1}^{1}}\|u_x^2\|_{B_{\infty,1}^{0}}\\
\leq& C\|u\|_{B_{\infty,1}^{1}}\|u_x\|_{B_{\infty,1}^{0}}\|u_x\|_{B_{\infty,\infty,1}^{0}}\leq C(\ln N)^2\|u_x\|_{B_{\infty,\infty,1}^{0}}.
\end{align*}
Thus, for any $t\in[0,T]$, we get
\begin{align}
\sum_{j\geq-1}2^j\|I_4\|_{L^\infty}
\leq&C\int_0^t\sum_{j\geq-1}2^j\|\widetilde{R_j}\circ y\|_{L^\infty}+\sum_{j\geq-1}2^j\|\left(\Delta_{j}G\right)\circ y\|_{L^\infty}\ d\tau\notag\\
\leq& CT(\ln N)^2\|u_x\|_{L^\infty_T\left(B^0_{\infty,\infty,1}\right)}+CTN^{-\frac 15}\ln N.\label{I4}
\end{align}
In order to estimate $\|u_x\|_{L^\infty_T\left(B^0_{\infty,\infty,1}\right)}$, we differentiate \eqref{eq1} once with respect to $x$ and apply $\Delta_{j}$ to it, thus obtaining
\begin{align*}
\left(\partial_{t}+u\partial_{x}\right)\Delta_{j}u_x=\Delta_{j}\widetilde{g}+u\partial_{x}\Delta_{j}u_x-\Delta_{j}\left(uu_{xx}\right),
\end{align*}
where $\widetilde{g}\triangleq-u_x^2-\partial_{x}F(u)$. Therefore, according to Lemma \ref{illlm2} and the similar proof of Lemma \ref{priori estimate}, we have
\begin{align*}
\|u_x\|_{L^\infty_T\left(B^0_{\infty,\infty,1}\right)}=&\sup_{t\in[0,T]}\sup_{j\geq-1}(j+2)^{1+\frac{1}{100}}\|\Delta_{j}u_x(t)\|_{L^\infty}\\
\leq&\sup_{j\geq-1}(j+2)^{1+\frac{1}{100}}\|\Delta_{j}u_{0x}\|_{L^\infty}+\int_0^T\sup_{j\geq-1}(j+2)^{1+\frac{1}{100}}\|\Delta_{j}\widetilde{g}(\tau)\|_{L^\infty}d\tau\\
&+\int_0^T\sup_{j\geq-1}(j+2)^{1+\frac{1}{100}}\|\left(u\partial_{x}\Delta_{j}u_x-\Delta_{j}\left(uu_{xx}\right)\right)(\tau)\|_{L^\infty}d\tau\\
\leq&\|u_{0x}\|_{B^0_{\infty,\infty,1}}+\int_0^T\|\widetilde{g}(\tau)\|_{B^0_{\infty,\infty,1}}d\tau+C\int_0^T\|u_x(\tau)\|_{{B^0_{\infty,1}}}\|u_x(\tau)\|_{B^0_{\infty,\infty,1}}d\tau.
\end{align*}
Since
\begin{align*}
\|\widetilde{g}\|_{B^0_{\infty,\infty,1}}\leq& C\left(\|u_x^2\|_{B^0_{\infty,\infty,1}}+\|\partial_{x}F_4(u)\|_{B^0_{\infty,\infty,1}}+\|\partial_{x}F_2(u)\|_{B^0_{\infty,\infty,1}}+\|\partial_{x}F_1(u)\|_{B^0_{\infty,\infty,1}}\right)\\
\leq& C\left(\|u_x^2\|_{B^0_{\infty,\infty,1}}+\|u^2\|_{B^0_{\infty,\infty,1}}\right)\\
\leq& C\left(\|u_x\|_{B^0_{\infty,1}}\|u_x\|_{B^0_{\infty,\infty,1}}+\|u^2\|_{B^1_{\infty,\infty}}\right)\\
\leq& C\left(\|u_x\|_{B^0_{\infty,1}}\|u_x\|_{B^0_{\infty,\infty,1}}+\|u^2\|_{C^{0,1}}\right)\\
\leq& C\left(\|u_x\|_{B^0_{\infty,1}}\|u_x\|_{B^0_{\infty,\infty,1}}+N^{-\frac 15}\right),
\end{align*}
applying \eqref{uoxBo..1} and \eqref{opposite} then yields
\begin{align*}
\|u_x\|_{L^\infty_T\left(B^0_{\infty,\infty,1}\right)}\leq&C N^{\frac{9}{10}+\frac{1}{100}}+CTN^{-\frac 15}+CT\left(\ln N\right)\|u_x\|_{L^\infty_T\left(B^0_{\infty,\infty,1}\right)}\\
\leq&C N^{\frac{9}{10}+\frac{1}{100}}+CN^{-\frac 12}\left(\ln N\right)\|u_x\|_{L^\infty_T\left(B^0_{\infty,\infty,1}\right)},
\end{align*}
which implies that $\|u_x\|_{L^\infty_T\left(B^0_{\infty,\infty,1}\right)}\leq C N^{\frac{9}{10}+\frac{1}{100}}$ when $N$ is large enough. Thus, we deduce from \eqref{I4} that
\begin{align*}
\sum_{j\geq-1}2^j\|I_4\|_{L^\infty}\leq &CT(\ln N)^2N^{\frac{9}{10}+\frac{1}{100}}+CTN^{-\frac 15}\ln N\\
\leq &C(\ln N)^2N^{\frac{2}{5}+\frac{1}{100}}+CN^{-\frac{7}{10}}\ln N.
\end{align*}
By $(1)-(3)$, \eqref{uoB1.1}, \eqref{uox2Bo.1} and \eqref{Deltaju}, we finally have
\begin{align*}
\|u(t)\|_{B^1_{\infty,1}}
\geq&\sum_{j\geq-1}2^j\left(t\|\Delta_{j}F_4(u_0)\|_{L^\infty}-\|\Delta_{j}u_0\|_{L^\infty}-\int_0^t\sum_{i=1}^4\|I_i\|_{L^\infty}d\tau\right)\\
\geq&t\|F_4(u_0)\|_{B^1_{\infty,1}}-\|u_0\|_{B^1_{\infty,1}}-Ct\left(N^{-\frac{1}{10}}\ln N+N^{-\frac 15}+(\ln N)^2N^{\frac{2}{5}+\frac{1}{100}}+N^{-\frac{7}{10}}\ln N\right)\\
\geq&Ct\|u_{0x}^2\|_{B^0_{\infty,1}}-C\left(N^{-\frac{1}{10}}+N^{-\frac{1}{10}-\frac 12}\ln N+N^{-\frac 15-\frac 12}+(\ln N)^2N^{-\frac {1}{10}+\frac{1}{100}}+N^{-\frac{7}{10}-\frac 12}\ln N\right)\\
\geq&\widetilde{C}tN^{\frac 35}-C.
\end{align*}
Therefore, for any $t\in\left[\frac 12 N^{-\frac{1}{2}},N^{-\frac 12}\right]$, we get
\begin{align*}
\|u(t)\|_{B^1_{\infty,1}}\geq\frac{\widetilde{C}}{2} N^{\frac{1}{10}}-C.
\end{align*}
Hence, we obtain
\begin{align*}
\sup_{t\in[0,\bar{T}]}\|u(t)\|_{B^1_{\infty,1}}\geq\frac{\widetilde{C}}{2} N^{\frac{1}{10}}-C>\ln N
\end{align*}
for $N$ large enough, which contradicts to the assumption \eqref{opposite}. So we gain the norm inflation and complete the proof of ill-posedness.
\end{proof}

Next, when $b=\frac 53$, we provide the ill-posedness in $B^{\frac 32}_{2,q}$ with $q\in(1,+\infty]$ in the sense of norm inflation.
\begin{theo}\label{ill2}
	Let $b=\frac 53$, $1<q\leq+\infty$. Then for any $N\in\mathbb{N}^+$ large enough, there exists a $u_0\in H^\infty(\mathbb{R})$ such that the following hold:\\
	(i) $\|u_{0}\|_{B^{3/2}_{2,q}} \leq \frac{C}{\ln N}\to 0$, as $N\to\infty$;\\
	(ii) There exists a unique solution $u \in \mathcal{C}\left([0,T);H^{\infty}(\mathbb{R})\right)$ to the Cauchy problem \eqref{eq1} with a time $T \leq \frac{1}{\ln N}$;\\
	(iii) $\|u\|_{L^\infty\left(\left[0,T\right);B^{3/2}_{2,q}\right)} \geq \ln N$.
\end{theo}
\begin{proof}
Let $\widetilde{\varphi}\in\mathcal{D}(\mathbb{R})$ be an even, non-zero, non-negative function such that $\widetilde{\varphi}\varphi=\widetilde{\varphi}$. For $N\in\mathbb{N}^+$ large enough, we define
\begin{align*}
u_0(x)=-\sum_{n=2}^{N}\dfrac{(\ln N)^{-1} h_n(x)}{2^{2n}n^{\frac{2}{1+q}}},
\end{align*}
where $\mathcal{F}(h_n)(\xi)=i2^{-n}\xi\widetilde{\varphi}(2^{-n}\xi)$. We see from the proof of Theorem 1.2 in \cite{GUOSharpill} that $u_0\in H^\infty(\mathbb{R})\cap B^{3/2}_{2,q}(\mathbb{R})$ is odd and satisfies
\begin{align*}
&\|u_0\|_{B^{3/2}_{2,q}}\leq C(\ln N)^{-1},\\
&\|u_0\|_{B^{3/2}_{2,1}}\leq CN^{\frac{q-1}{q+1}},\\
&u_{0x}(0)\geq C(\ln N)^{-1}N^{\frac{q-1}{q+1}}>0.
\end{align*}
Thus, the proof of Theorem \ref{the1} gives that there is a lifespan $T_N\sim N^{-\frac{q-1}{q+1}}$ and a solution $u\in\mathcal{C}\left([0,T_N);H^{\infty}(\mathbb{R})\right)$ such that $\|u\|_{L^\infty\left(\left[0,T_N\right);B^{3/2}_{2,q}\right)}<\infty$.

Next, we denote $\widetilde{T}\triangleq(\ln N)^{-1}$ and consider $\|u\|_{L^\infty\left(\left[0,\widetilde{T}\right);B^{3/2}_{2,q}\right)}$ for fixed $N$ large enough. If $\|u\|_{L^\infty\left(\left[0,\widetilde{T}\right);B^{3/2}_{2,q}\right)}\geq\ln N$, then Theorem \ref{ill2} is proved. Otherwise, if we assume that
\begin{align}\label{opposite2}
\|u\|_{L^\infty\left(\left[0,\widetilde{T}\right);B^{3/2}_{2,q}\right)}<\ln N,
\end{align}
then similar to the proof of Proposition \ref{blow up criterion}, we have
\begin{align*}
\log\left(e+\|u(t)\|_{H^{s}}\right)\leq\left(\log\left(e+\|u_{0}\|_{H^{s}}\right)+Ct\right)e^{C\int_{0}^{t}\|u\|_{B_{2,\infty}^{3/2}}d\tau},
\end{align*}
which implies $u\in L^\infty([0,\widetilde{T});H^{s}(\mathbb{R}))\ (s>\frac 32)$. Since $u_0$ is odd, we can easily deduce from \eqref{eq1} that $u$ is odd and $u(t,0)=u_{xx}(t,0)=0$. By differentiating \eqref{eq1} once with respect to $x$ and setting $x=0$, we get
\begin{align*}
\frac{d}{dt}u_x(t,0)=-u_x(t,0)^2-\partial_{x}F(u)(t,0)
\geq\frac 23 u_x(t,0)^2-C_{\alpha,\beta}\left(\|u(t)\|^2_{L^{\infty}}+\|u_x(t)\|^2_{L^{\infty}}\right).
\end{align*}
Taking advantage of Lemma \ref{log} and \eqref{opposite2}, we see that
\begin{align*}
\|u\|_{W^{1,\infty}}\leq C\ln N,
\end{align*}
thus obtaining
\begin{align}\label{ill2ineq}
\frac{d}{dt}u_x(t,0)\geq\frac 23 u_x(t,0)^2-C_{\alpha,\beta}(\ln N)^2,\ \forall\ t\in[0,\widetilde{T}).
\end{align}
Since $q>1$, for $N>0$ large enough, we have
\begin{align*}
u_{0x}(0)\geq C(\ln N)^{-1}N^{\frac{q-1}{q+1}}\gg\frac{\sqrt{C_{\alpha,\beta}}\ln N}{\sqrt{2/3}}>0.
\end{align*}
Denote $C_0\triangleq\frac{\sqrt{\frac 23}u_{0x}(0)-\sqrt{C_{\alpha,\beta}}\ln N}{\sqrt{\frac 23}u_{0x}(0)+\sqrt{C_{\alpha,\beta}}\ln N}\in(0,1)$, then solving \eqref{ill2ineq} gives
\begin{align*}
u_x(t,0)\geq\dfrac{\sqrt{C_{\alpha,\beta}}\ln N\left(1+C_0\exp\left(2\sqrt{\frac 23 C_{\alpha,\beta}}\ln N\cdot t\right)\right)}{\sqrt{\frac 23}\left(1-C_0\exp\left(2\sqrt{\frac 23 C_{\alpha,\beta}}\ln N\cdot t\right)\right)},
\end{align*}
which implies that there exists
\begin{align*}
0<T_0\leq-\ln C_0\cdot\dfrac{1}{2\sqrt{\frac 23 C_{\alpha,\beta}}\ln N}<\widetilde{T} 
\end{align*}
for $N>0$ large enough such that $u$ blows up in finite time $T_0$. Therefore, according to Proposition \ref{blow up criterion}, we have
\begin{align*}
\|u\|_{L^\infty\left([0,\widetilde{T});B^{3/2}_{2,q}\right)}\geq C\|u\|_{L^\infty\left([0,\widetilde{T});B^{1}_{\infty,\infty}\right)}=+\infty,
\end{align*}
which contradicts to the hypothesis \eqref{opposite2}. Hence, we get $\|u\|_{L^\infty\left(\left[0,\widetilde{T}\right);B^{3/2}_{2,q}\right)}\geq\ln N$ and complete the proof.
\end{proof}
\smallskip
\noindent\textbf{Acknowledgments.} This work was supported by the National Natural Science Foundation of China (No.12171493).


\phantomsection
\addcontentsline{toc}{section}{\refname}
\bibliographystyle{abbrv} 
\bibliography{Feneref}
	
\end{document}